\documentclass[a4paper, reqno, 11pt]{amsart}
\makeatletter

\@addtoreset{equation}{section}
\makeatother
\usepackage{color}
\usepackage{graphicx}
\usepackage{amsmath}
\usepackage{amssymb}
\usepackage{float}
\usepackage{amsthm}
\newtheorem{Thm}{Theorem}[section]

\newtheorem{Lem}[Thm]{Lemma}
\newtheorem{Prop}[Thm]{Proposition}

\newtheorem{Cor}[Thm]{Corollary}

\theoremstyle{definition}

\newtheorem{Def}[Thm]{Definition}
\newtheorem{Rem}[Thm]{Remark}

\newcommand{\R}{\mathbb{R}}
\newcommand{\E}{\mathcal{E}}
\newcommand{\F}{\mathcal{F}}
\newcommand{\M}{\mathcal{M}}
\newcommand{\V}{V_{\epsilon}(t)}
\newcommand{\ca}{\textup{Cap}}

\begin{document}

\title[Wiener sausage on Dirichlet space]{Long time behavior of the volume of the Wiener sausage on Dirichlet spaces} 
\author{Kazuki Okamura} 
\address{School of General Education, Shinshu University, 3-1-1, Asahi, Matsumoto, Nagano, 390-8621, JAPAN.} 
\email{kazukio@shinshu-u.ac.jp} 
\maketitle

\begin{abstract}
In the present paper, we consider long time behaviors of the volume of the Wiener sausage on Dirichlet spaces. 
We focus on the volume of the Wiener sausage for diffusion processes on metric measure spaces other than the Euclid space equipped with the Lebesgue measure. 
We obtain the growth rate of the expectations and almost sure behaviors of the volumes of the Wiener sausages on metric measure Dirichlet spaces satisfying Ahlfors regularity and sub-Gaussian heat kernel estimates. 
We show that the growth rate of the expectations on a bounded modification of the Euclidian space is identical with the one on the Euclidian space equipped with the Lebesgue measure.  
We give an example of a metric measure Dirichlet space on which a scaled of the means fluctuates.
\end{abstract}

\section{Introduction and Main results}

The Wiener sausage is the range of a ball whose center moves along the trajectory of a Markov process. 
This is a simple example of a non-Markov functional of the Markov process. 
It is also related with analysis and mathematical physics, specifically, the study of heat conduction, spectral properties of random Schr\"odinger operators, and Bose-Einstein condensation. 
The first result concerning the volume of the Wiener sausage is by Spitzer \cite{Sp64}. 
He considered the long time behavior of the volume  of the Wiener sausage on Euclid spaces and his results are strongly related to heat conduction. 
The Laplace transform of the volume of the Wiener sausage is related to spectral properties of random Schr\"odinger operators. 
Donsker-Varadhan \cite{DV75} considered the Laplace transform of the volume of the Wiener sausage and answered a question by Kac-Luttinger \cite{KL74} concerning the Bose-Einstein condensation. 

In the present paper, we consider the long time behavior of the volume of the Wiener sausage on metric measure Dirichlet spaces satisfying (sub-)Gausian heat kernel estimates. 
There have been many results for the standard Brownian motion on the Euclid spaces. 
Here we focus on the volume of the Wiener sausage for diffusion processes on metric measure spaces other than the Euclid space equipped with the Lebesgue measure. 

We review several known results for diffusion processes on metric measure spaces other than the Euclid space equipped with the Lebesgue measure.     
 Chavel-Feldman \cite{CF86-1, CF86-2, CF86-3} considered the volume of the Wiener sausage for Brownian motion on Riemannian manifolds.   
\cite{CF86-1} shows radial asymptotic results (i.e. radius $\epsilon \to 0$) on hyperbolic 3-spaces, and a time asymptotic result (i.e. time $t \to \infty$) on Riemannian symmetric spaces of non-positive curvature. 
\cite{CF86-2} shows radial asymptotic results  on complete Riemannian manifolds for the dimension $d \ge 3$. 
\cite{CF86-3} shows radial asymptotic results for the Wiener sausage of reflected Brownian motion on a domain in $\R^d, d \ge 2$.  
Sznitman \cite{Sz89} obtained a time asymptotic result of negative exponentials of Brownian bridge on hyperbolic space which is similar to the result by \cite{DV75}. 
Furthermore, in \cite{Sz90} he obtains a Donsker-Varadhan type result for Brownian motion on nilpotent Lie groups.
Chavel-Feldman-Rosen \cite{CFR91} obtained a second order radial asymptotic result for $2$-dimensional Riemannian manifold, extending Le Gall's expansion \cite[Theorem 2.1]{Le88} in $\R^2$.
Gibson-Pivarski \cite{GP15} obtained a time asymptotic result similar to \cite{DV75} for diffusions on local Dirichlet spaces satisfying a sub-Gaussian heat kernel estimate. 
Bass-Kumagai \cite{BK00} showed the law of the iterated logarithms (LILs) of the ranges of a class of symmetric diffusion 
processes on metric measure spaces. 
Recently, Kim-Kumagai-Wang \cite{KKW17} showed the LILs of the ranges of a class of symmetric jump processes on metric measure spaces. 
Here the range is the volume of the trace of a sample path up to a fixed time. 
The results of \cite{BK00, KKW17, GP15} are based on heat kernel estimates. 

Our results are time asymptotics for the volume of the Wiener sausage on metric measure Dirichlet spaces. 
First, we obtain growth rate of the means and almost sure behaviors on certain classes of metric measure Dirichlet spaces. 
Our classes contain metric measure Dirichlet spaces satisfying Ahlfors regularity and sub-Gaussian heat kernel estimates. 
Second, we show that the growth rate of the expectations on a modification of the Euclidian space equipped with the Lebesgue measure is identical with those of the Euclidian space equipped with the Lebesgue measure.   
Third, we give an example of a metric measure Dirichlet space on which a scaled sequence of the expectations fluctuates.  

For the Brownian motion on the Euclid spaces, by using the Brownian scaling, time asymptotic results can be derived from radial asymptotic results.  
However, in this case, we cannot use the spatial homogeneity and the scaling of the Euclid space and the Brownian motion, which are used to show the time asymptotic results in \cite{Sp64} and \cite{CF86-1}. 
Our proofs are based on two estimates for the probability for the hitting time to an open ball appearing in \cite[(10) and (11)]{CF86-2}, which are stated in Lemma \ref{lem-fund} below.  

The range of random walk is the discrete object corresponding to the volume of the Wiener sausage. 
There are differences between the ranges of random walks on graphs and the volumes of the Wiener sausages of diffusions on metric measure Dirichlet spaces. 
The results in the present paper correspond to those for the range of a simple random walk on an infinite connected simple graph, which were obtained by the author \cite{O14, Opre}.  
However, our proofs are different from those of \cite{O14, Opre}.  
The last exit decomposition in \cite{O14, Opre}, whose papers deal with the discrete case, is not  applicable to this framework at least in direct manners.   
For each $n$, the range of random walk up to time $n$ is always smaller than or equal to $n+1$, however, $\V$ is unbounded for each $t > 0$.

\subsection{Framework and main results}

Let $(M, d)$ be a non-compact connected complete separable metric space such that every open ball is relatively compact. 
Let $\mu$ be a Borel measure on $M$ such that for every relatively compact open subset $U$ of $M$, $0 < \mu(U) \le \mu(\overline{U}) < +\infty$. 
Here $\overline{U}$ is the closure of $U$. 
Let $B(x,r) := \{y \in M : d(x,y) < r\}$, $\overline{B}(x,r) := \{y \in M : d(x,y) \le r\}$ and $\overline{\partial}B(x,r) := \{y \in M : d(x,y) = r\}$. 
We assume that 
\begin{equation}\label{non-emp-bdry}
\overline{\partial}B(x,r) \ne \emptyset,  \textup{ for every $x \in M$ and every $r \ge 0$.} 
\end{equation} 
and furthermore 
\begin{equation}\label{meas-vanish-bdry}
\mu(\overline{\partial}B(x,r)) = 0, \textup{ for every $x \in M$ and every $r \ge 0$.} 
\end{equation} 
 
We follow Fukushima-Oshima-Takeda's book \cite{FOT11} for terminologies of Dirichlet forms. 
Let $(\E, \F)$ be a strongly-local regular symmetric {\it conservative} 
Dirichlet form on $L^2 (M, \mu)$ and $(X_t, P^x)$ be the associated Hunt process.  
Assume that there exists the heat kernel of the semigroup associated with $(\E, \F)$. 
We furthermore assume that $p(t,x,y)$ is jointly-continuous with respect to $(t,x,y)$ and 
\[ E^x [f(X_t)] = \int_M p(t,x,y) \mu(dy), \ \forall x \in M, \ \forall f \in L^2 (M, \mu).\]  
We call a quintuple  $(M, d, \mu, \E, \F)$ a {\it metric measure Dirichlet space}. 

For a Borel measurable subset $B$ of $M$, 
let $$T_{B} := \inf\{t > 0 : X_t \in B\}.$$ 
Let $\tau_D$ be the first exit time from a Borel measurable subset $D \subset M$, that is, 
$\tau_D := T_{M \setminus D}$.  
Let the volume of the Wiener sausage: 
\[ \V := \mu\left(\bigcup_{s \in [0,t]} B(X_s, \epsilon) \right). \]

We write $f \asymp g$ if there are two constants $c$ and $C$ such that $cg(x) \le f(x) \le Cg(x)$ for every $x$.   

\begin{Thm}[Growth rates for means]\label{thm-fractal}
Let $\epsilon > 0$ and $a \in (0, 1)$.   
Assume that the following two conditions hold: \\
(i) There is a constant $c_0 > 0$ such that   
\begin{equation}\label{ass-vg} 
\sup_{x \in M} \mu(B(x, a\epsilon)) \le c_0 \inf_{x \in M} \mu(B(x, a\epsilon)). 
\end{equation} 
(ii) There are an increasing function $f(t)$ and constants $c_1, c_2 > 0$ such that   
\begin{equation*}
\lim_{t \to \infty} \frac{f(t)}{t} = 0.  
\end{equation*}  
\begin{equation}\label{ass-growth-hk}  
c_1 \le \liminf_{t \to \infty} \frac{1}{f(t)}\int_0^t p(s,x,y) ds \le \limsup_{t \to \infty} \frac{1}{f(t)}\int_0^t p(s,x,y) ds \le  c_2 
\end{equation} 
 holds for  
 every $x, y \in M$ satisfying that $(1-a)\epsilon \le d(x,y) \le (1+a)\epsilon$. 
Then, we have that for every $x \in M$, 
\[ \frac{c_0}{c_2} \le \liminf_{t \to \infty} \frac{E^x [\V]}{t / f(t)} \le \limsup_{t \to \infty} \frac{E^x [\V]}{t  / f(t)} \le \frac{c_0}{c_1}.  \]
\end{Thm}

The constants $c_i, i = 0,1,2$, and the function $f$ depend on $\epsilon > 0$ and $a \in (0,1)$ both. 
However in several cases we can choose $f$ as a function independent from $\epsilon > 0$ and $a \in (0,1)$ both and $\lim_{a \to 0} c_1 (a) = \lim_{a \to 0} c_2(a)$. 
Assume that $M = \R^d$ and $(X_t)_t$ is the standard Brownian motion.  
If $d \ge 3$, then, we can let $f(t) = 1$, $c_0(a) = 1$ and $\lim_{a \to 0} c_1 (a) = \lim_{a \to 0} c_2(a) = G(x,y)$, where $x$ and $y$ are points such that $d(x,y) = \epsilon$.  
Therefore, 
\begin{equation}\label{Sp-d3} 
\lim_{t \to \infty} \frac{E^x [\V]}{t} = \frac{1}{G(x,y)} = \textup{Cap}(B(0,\epsilon)). 
\end{equation} 

If $d = 2$, then, we can let $f(t) = \log t$, $c_0(a) = 1$ and $\lim_{a \to 0} c_1 (a) = \lim_{a \to 0} c_2(a) = 1/(2\pi)$.  
Therefore, 
\begin{equation}\label{Sp-d2}  
\lim_{t \to \infty} \frac{E^x [\V]}{t/\log t} = 2\pi. 
\end{equation}

If $d = 1$, then, we can let $f(t) = t^{1/2}$, $c_0(a) = 1$ and $\lim_{a \to 0} c_1 (a) = \lim_{a \to 0} c_2(a) = (2\pi)^{-1/2}$.     
Therefore, 
\[ \lim_{t \to \infty} \frac{E^x [\V]}{t^{1/2}} = \sqrt{2\pi}. \]

\eqref{Sp-d3} and \eqref{Sp-d2} reproduce the first-order expansions of Spitzer's results \cite[Theorems 1 and 2]{Sp64} respectively.

If we know the asymptotic of mean of $\V$, then, it is very natural to investigate {\it almost sure} behaviors of $\V$. 
We say that $\textup{Vol}(V; \alpha_1, \alpha_2)$ holds if  
there exist four positive constants $\alpha_1 < \alpha_2$ and $C_1 \le C_2$, and a strictly increasing function $V$ such that 
\begin{equation}\label{V-alpha} 
C_1 \left(\frac{R}{r}\right)^{\alpha_1} \le \frac{V(R)}{V(r)} \le C_2 \left(\frac{R}{r}\right)^{\alpha_2}, \ 0 < r < R, 
\end{equation}
and furthermore there exists a constant $c > 1$  such that 
\begin{equation}\label{UVG}
c^{-1}V(r) \le \mu(B(x,r)) \le cV(r), \ \ x \in M,  r > 0.
\end{equation}

If $\textup{Vol}(V; \alpha_1, \alpha_2)$ holds for a certain $V$ satisfying \eqref{V-alpha} and $\alpha_1 = \alpha_2$, then, we say that $\textup{Vol}(\alpha)$ holds for $\alpha = \alpha_1 = \alpha_2$. 

Now we state assumptions for the heat kernel by following the terminologies in Grigor'yan-Telcs \cite{GT12}. 
We say that HK$(\phi; \beta_1, \beta_2)$ holds if there exist four positive constants $c_5, c_6, c_7$ and $c_8$ and a function $\phi$  such that 
\begin{equation}\label{HK-lower} 
\frac{c_5}{ \mu(B(x,\phi^{-1}(t)))} \le p(t,x,y), \ d(x,y) \le c_6\phi^{-1}(t), 
\end{equation}
and,   
\begin{equation}\label{HK-upper} 
p(t,x,y) \le  \frac{c_7 \exp\left(-c_8 \Psi(d(x,y), t)  \right)}{ \mu(B(x,\phi^{-1}(t)))}, \ x,y \in M, t > 0. 
\end{equation} 
where we let 
\[ \Psi(r, t) := \sup_{s > 0} \left(\frac{r}{s} - \frac{t}{\phi(s)}\right), \]  
and where $\phi(\cdot)$ is a strictly increasing continuous function on $(0,+\infty)$ 
such that  there exist four constants $1 < \beta_1 \le \beta_2$ and $0 < C_3 \le C_4$ such that 
\begin{equation}\label{phi-beta} 
C_3 \left(\frac{R}{r}\right)^{\beta_1} \le \frac{\phi(R)}{\phi(r)} \le C_4 \left(\frac{R}{r}\right)^{\beta_2}, \ 0 < r < R. 
\end{equation} 

We say that $\textup{FHK}(\phi; \beta_1, \beta_2)$  holds if there exist four positive constants $c_5, c_6, c_7$ and $c_8$  and a function $\phi$ on $(0, +\infty)$ such that 
\[ \frac{c_5 \exp\left(-c_6 \Psi(d(x,y), t)  \right) }{\mu(B(x,\phi^{-1}(t)))} \le p(t,x,y) \le  \frac{c_7 \exp\left(-c_8 \Psi(d(x,y), t)  \right)}{\mu(B(x,\phi^{-1}(t)))}, \ t > 0, x, y \in M, \]
where $\Psi$ and $\phi$ are defined and characterized as above.  

By the definition of $\Psi$, it is easy to see that there exists a positive constant $C(\beta_1, \beta_2)$ such that $\sup_{t > 0} \Psi(\phi^{-1}(t), t) \le C(\beta_1, \beta_2)$, and hence, 
$\textup{FHK}(\phi; \beta_1, \beta_2)$ implies HK$(\phi; \beta_1, \beta_2)$.

If HK$(\phi; \beta_1, \beta_2)$ (resp. $\textup{FHK}(\phi; \beta_1, \beta_2)$) holds for a certain $\phi$ satisfying \eqref{phi-beta}, and furthermore $\beta_1 = \beta_2$, 
then, we say that HK$(\beta)$ (resp. $\textup{FHK}(\beta)$) holds for $\beta = \beta_1 = \beta_2$. 
By Barlow-Grigor'yan-Kumagai \cite[Proposition 5.2 (i) and (iii)]{BGK12}, we have that if the metric $d$ is geodesic, then HK$(\beta)$ and the volume doubling are equivalent to FHK$(\beta)$. 

For almost sure behaviors of $\V$, we have the following. 
\begin{Thm}[Almost sure behaviors]\label{pre-SLLN}
Assume that $\textup{Vol}(V; \alpha_1, \alpha_2)$ and HK$(\phi; \beta_1, \beta_2)$ hold. 
Let  
\begin{equation}\label{def-f} 
f(t) := \int_1^{t} \frac{ds}{V(\phi^{-1}(s))}, \ t > 1.  
\end{equation} 
Then, \\
(i) There exist two (non-random) constants $c_1, c_2 \in [0, \infty)$  such that for every $x \in M$, 
\begin{equation}\label{non-str-rec-inf}
\liminf_{t \to \infty} \frac{\V}{\min\{V(\phi^{-1}(t/\log\log t)), t / f(t/\log\log t)\}} = c_1, \textup{ $P^x$-a.s.}  
\end{equation}  
\begin{equation}\label{non-str-rec-sup} 
\limsup_{t \to \infty} \frac{\V}{t / f(t/\log\log t)} = c_2, \textup{ $P^x$-a.s.} 
\end{equation} 
(ii) If 
\begin{equation}\label{ratio}
\limsup_{t \to \infty} \frac{f(t/\log\log t)}{f(t)} > 0,
\end{equation} 
then, $c_2 > 0$.\\ 
(iii) If  $\textup{FHK}(\phi; \beta_1, \beta_2)$ hold and furthermore $\alpha_2 < \beta_1$, then, $c_1 > 0$ and $c_2 > 0$. 
\end{Thm}

We remark that the function $f(t)$ in \eqref{def-f} comes from the on-diagonal estimate of the heat kernel in \eqref{HK-lower} and \eqref{HK-upper}.
We also remark that the constants $c_1$ and $c_2$ in the above theorem are independent from the choice of $x \in M$. 
We do not have an example satisfying that $c_1 = 0$, and we conjecture that $c_1 > 0$ holds under the assumptions $\textup{Vol}(V; \alpha_1, \alpha_2)$  and HK$(\phi; \beta_1, \beta_2)$.  
Theorems \ref{thm-fractal} and \ref{pre-SLLN} are applicable to several classes of diffusions on fractal graphs and Riemannian manifolds including fractal-like manifolds. 
See Subsection 2.1 for details.


Now we define a notion of a bounded modification of an Euclid space equipped with the Lebesgue measure. 
We do not modify the metric structure of an Euclid space. 

\begin{Def}[Bounded modification]\label{bddm-def}
We say that $(\R^d, d, \widetilde\mu, \widetilde X_t, \widetilde P^x)$ is a {\it bounded modification} of a metric measure Dirichlet space $(\R^d, d, \mu, X_t, P^x)$ 
if there exists a bounded domain $D$ of $\mathbb{R}^d$ such that \\ 
(M1) For every Borel subset $A$ of $\R^d \setminus D$, $\widetilde\mu(A) = \mu(A)$.\\
(M2) For every $x \in \R^d \setminus D$, the law of $\widetilde X_{\cdot \wedge \widetilde T_D}$ under $\widetilde P^x$ is identical with the law of $X_{\cdot \wedge T_D}$ under $P^x$, where $\widetilde T_D$ and $T_D$ are the first hitting times of $\widetilde X$ and $X$ to $D$, respectively. 
\end{Def}

\begin{Thm}[Behaviors on bounded modifications]\label{thm-finite}
(i) Let $d \ge 3$. 
Assume that $(\R^d, d, \mu, X_t, P^x)$ is a bounded modification of the quintuple of $\R^d$, the Euclid distance $d$, the Lebesgue measure $\mu$,   
and the standard Brownian motion $(X_t, P^x)$ on $\R^d$. 
We furthermore assume that FHK$(2)$ holds for $(\R^d, d, \mu, X_t, P^x)$. 
Then, for every $x \in M$,  
\[ \lim_{t \to \infty} \frac{E^x [\V]}{t} = \ca_{\R^d}\left(\overline{B}(0,\epsilon)\right). \]
Here $\ca_{\R^d}$ is Newtonian capacity. \\
(ii) If $(\R^2, d, \mu, X_t, P^x)$ is a bounded modification of the quintuple of $\R^2$, the Euclid distance $d$, the Lebesgue measure and the standard Brownian motion.
We furthermore assume FHK$(2)$ holds for $(\R^2, d,  \mu, X_t, P^x)$. 
then, for every $x \in \R^2$,  
\[ \lim_{t \to \infty} \frac{ E^x [\V] }{t / \log t}  = 2\pi.  \]
\end{Thm}

We give a definition of rough isometries by following Barlow-Bass-Kumagai \cite[Definition 2.20]{BBK06}. 
In the context of Riemannian manifolds, the notion of {\it rough isometry} is introduced by Kanai \cite{Ka85}. 

We say that two metric measure space $(M_1, d_1, \mu_1)$ and $(M_2, d_2, \mu_2)$ are {\it roughly isometric} 
if there exist a map $\varphi : M_1 \to M_2$ and three positive constants $c_1, c_2$ and $c_3$ satisfying the following three conditions: \\
(i) \[ M_2 = \bigcup_{x \in M_1} B_{d_2} (\varphi(x), c_1). \]
(ii) \[ c_2^{-1} (d_1 (x, y) - c_1)  \le d_2 (\varphi(x), \varphi(y)) \le c_2 (d_1 (x, y) + c_1), \ x, y \in M_1. \]
(iii) \[ c_3^{-1} \mu_1 (B_{d_1}(x, c_1)) \le \mu_2 (B_{d_2}(\varphi(x), c_1)) \le c_3 \mu_1 (B_{d_1}(x, c_1)), \ x \in M_1. \]

By \cite{BBK06}, 
it is known that the estimate FHK$(2)$ is {\it  stable under rough isometries} between metric measure Dirichlet spaces, given suitable local regularity of the two spaces. 
The assumption  FHK$(2)$ prevents ``singular" behaviors of $(X_t)_t$ when it enters $D$. 

\begin{Thm}[Fluctuations]\label{thm-fluc-1}
Let $d \ge 3$. 
Fix $\epsilon > 0$. 
Then, there is a metric measure Dirichlet space $(\R^d, d, \mu, \E, \F)$ such that the following conditions (a) and (b) hold: \\
(a) $(\R^d, d, \mu)$ is roughly isometric to the Euclid space equipped with the Lebesgue measure.\\ 
(b) 
\[ \liminf_{t \to \infty} \frac{E^0 [\V]}{t} \le \ca_{\R^d}\left(\overline{B}(0,\epsilon)\right)  < 2^{(d-2)/2} \ca_{\R^d}\left(\overline{B}(0,\epsilon)\right) \le \limsup_{t \to \infty} \frac{E^0 [\V]}{t}.  \]
In particular, the limit $\displaystyle \lim_{t \to \infty} \frac{E^0 [\V]}{t} $ does \textup{not} exist. 
\end{Thm}

We will show this by using Theorem \ref{thm-finite}.  
The following is a more detailed result for $(\V)_t$ of bounded modifications of $\R^d$ in high dimensions.   

\begin{Thm}[Behaviors on bounded modifications]\label{Lv2-mod}
Let $d \ge 6$.
Assume that $\mathcal{M} = (\R^d, d, \mu, X_t, P^x)$ is a bounded modification of the quintuple of $\R^d$, the Euclid distance $d$, the Lebesgue measure $\mu$ and the standard Brownian motion $(X_t, P^x)$ on $\R^d$.  
Let $E_{\mathcal{M}}^x$ be the expectation with respect to $P^x$. 
We furthermore assume that FHK$(2)$ holds and there exists a constant $C$ such that for every $R > 0$, 
\[ \sup_{x \in \R^d} \mu(B(x,R)) \le CR^d.\]    
Then, for every $x \in \R^d$,  
\[ \lim_{t \to \infty} \left(E_{\mathcal{M}}^x [\V] - t\ca_{\R^d}\left(\overline{B}(0,\epsilon)\right)\right) \textup{ exists and is finite.} \]
\end{Thm}

The assumption that $d \ge 6$ in the above theorem is due to \cite{Sp64}. 

We do not know the value of the following limit 
\[ \lim_{t \to \infty} \left(E_{\mathcal{M}}^x [\V] - E_{\textup{BM}}^x [\V]\right). \]  

The organization of the rest of the present paper is as follows. 
In Section 2, we deal with growth rates of means and almost sure behaviors and show Theorems \ref{thm-fractal} and \ref{pre-SLLN}.  
In Section 3, we consider the behavior of process on bounded modifications and show Theorem \ref{thm-finite}. 
In Section 4, we deal with the case that a scaled sequence of the means fluctuates and show Theorem \ref{thm-fluc-1}.  
In Section 5, we consider more detailed behavior of process on bounded modifications and show Theorem \ref{Lv2-mod}. 


\section{Growth rates of means and almost sure behaviors} 

We first remark that by Fubini's theorem, 
\begin{equation}\label{base}
E^x [\V] = \int_M P^x (T_{B(y, \epsilon)} \le t) d\mu(y). 
\end{equation}

Contrary to the Brownian motion on Euclid spaces, we cannot expect in general that 
\[ P^x (T_{B(y, \epsilon)} \le t) = P^y (T_{B(x, \epsilon)} \le t). \]
We give upper and lower bounds for $P^x (T_{B(y, \epsilon)} \le t)$. 
 
\begin{Lem}[{\cite[(10) and  (11)]{CF86-2}}]\label{lem-fund}  
Let $t, T, \epsilon > 0$.  
Let $d(x,y) > \epsilon$. 
Then, \\
(i) (upper bound) 
For every $\eta > 0$,  
\[ \int_{0}^{t+T} P^x (X_{s} \in B(y, \eta)) ds 
\ge P^x (T_{B(y, \epsilon)} \le t) \inf_{w \in \overline{\partial}B(y, \epsilon)}  \int_{0}^{T} P^w (X_{s}  \in B(y, \eta)) ds. \]  
(ii) (lower bound) 
For every $a \in (0,1)$, 
\[ \int_{0}^{t}  P^x (X_{s}  \in B(y, a\epsilon ))  ds 
\le P^x (T_{B(y, \epsilon)} \le t) \sup_{w  \in \overline{\partial}B(y, \epsilon)}  \int_{0}^{t} P^w (X_{s}  \in B(y, a\epsilon ))  ds. \]  
\end{Lem}

These inequalities are easily seen by using the strong Markov property. Lemma \ref{lem-fund} and \eqref{base} enable us to give upper and lower bounds for $E[\V]$. 

Now we show Theorem \ref{thm-fractal}. 
By \eqref{ass-vg}, it suffices to show the following:  

\begin{Lem}\label{lem-fund-rough}
Let $c_1$ and $c_2$ be constants in \eqref{ass-growth-hk}. 
Then, for every $\epsilon > 0$, $a \in (0,1)$ and every $o \in M$, \\ 
(i) 
\[ \limsup_{t \to \infty} \frac{E^o [\V]}{t/f(t)} \le \frac{1}{c_1} \frac{\sup_{x \in M} \mu(B(x, a\epsilon))}{\inf_{x \in M}  \mu(B(x, a\epsilon))}.\] 
(ii)  
\[ \liminf_{t \to \infty} \frac{E^o [\V]}{t/f(t)} \ge \frac{1}{c_2} \frac{\inf_{x \in M} \mu(B(x, a\epsilon))}{\sup_{x \in M}  \mu(B(x, a\epsilon))}.\]
\end{Lem}

\begin{proof}
(i) Let $\delta \in (0,1)$.   
Recall \eqref{non-emp-bdry}.
Applying Lemma \ref{lem-fund} (i) to the case that $\eta = a\epsilon$ and $T = \delta t$, we have that 
\begin{align*}
E_{M}^o [\V] &= \int_{M} P^o (T_{B(y,\epsilon)} \le t) d\mu(y) \\
&\le \int_{M}   \frac{\int_{0}^{(1+\delta)t} P^o (X_s \in B(y, a\epsilon )) ds }{\inf_{w \in \overline{\partial}B(y, \epsilon)}  \int_{0}^{t} P^w (X_s \in B(y, a\epsilon )) ds}   d\mu(y)  \\
&\le  \frac{\int_{M} \int_{0}^{(1+\delta)t}   P^o (X_s \in B(y, a\epsilon)) ds d\mu(y)}{\inf_{y, w \in M; d(y, w) = \epsilon}  \int_{0}^{t} P^w (X_s \in B(y, a\epsilon )) ds} \\
&\le  \frac{\int_{0}^{(1+\delta)t} E^o \left[\mu\left(B(X_s, a\epsilon)\right)\right] ds}{\inf_{y, w \in M; d(y, w) = \epsilon}  \int_{0}^{t} P^w (X_s \in B(y, a\epsilon))ds}. \\
&\le  \frac{(1+\delta)t \sup_{z \in M}  \mu(B(z, a\epsilon))}{\inf_{y, w \in M; d(y, w) = \epsilon}  \int_{0}^{t} P^w (X_s \in B(y, a\epsilon)) ds}. \\
&\le (1+\delta) \frac{\sup_{z \in M}  \mu(B(z, a\epsilon))}{\inf_{z \in M}  \mu(B(z, a\epsilon))} \sup_{w,z \in M; (1-a)\epsilon \le d(w,z) \le (1+a)\epsilon}  \frac{t}{\int_{0}^{t} p(s, w, z)ds}.  \\
&\le \frac{1+\delta}{c_0} \sup_{w,z \in M; (1-a)\epsilon \le d(w,z) \le (1+a)\epsilon}  \frac{t}{\int_{0}^{t} p(s, w, z)ds},
\end{align*} 
where $c_0$ is the constant in \eqref{ass-vg}. 

Recall \eqref{base}. 
Since $\delta$ is taken arbitrarily,  it holds that 
\begin{align}\label{upper-mean} 
E_{M}^o [\V]
\le  \frac{1}{c_0} \sup_{w,z \in M; (1-a)\epsilon \le d(w,z) \le (1+a)\epsilon}  \frac{t}{\int_{0}^{t} p(s, w, z)ds}. 
\end{align}   

On the other hand, by \eqref{ass-growth-hk},   
\[  \limsup_{t \to \infty} \sup_{w,z \in M; (1-a)\epsilon \le d(w,z) \le (1+a)\epsilon} \frac{f(t)}{\int_0^t p(s,w,z) ds} \le \frac{1}{c_1}. \]
By this and \eqref{upper-mean}, 
we have assertion (i). 

(ii) By \eqref{non-emp-bdry}, \eqref{base} and Lemma \ref{lem-fund}, 
\begin{align}\label{lower-mean} 
E^o [\V] &= \int_{M} P^o (T_{B(y,\epsilon)} \le t) d\mu(y) \notag\\ 
&\ge \int_{M} \frac{\int_{0}^{t} P^x ( X_s \in B(y, a\epsilon ) ) ds}{\sup_{w \in \overline{\partial}B(y, \epsilon)}  \int_{0}^{t} P^w ( X_s \in B(y, a\epsilon) )  ds} d\mu(y)  \notag\\ 
&\ge \frac{\int_{0}^{t} \int_{M} \mu\left(B(z, a\epsilon)\right) p(s,x,z) d\mu(z) ds}{\sup_{w,y \in M; d(w, y) = \epsilon} \int_{0}^{t}  P^w ( X_s \in B(y, a\epsilon ) )   ds}  \notag\\
&\ge \frac{t \inf_{z \in M} \mu(B(z, a\epsilon))}{\sup_{w,y \in M; d(w, y) = \epsilon}  \int_{0}^{t}  P^w ( X_s \in B(y, a\epsilon ) )  ds}  \notag\\
&\ge \frac{\inf_{z \in M} \mu(B(z, a\epsilon))}{\sup_{z \in M } \mu(B(z, a\epsilon))} \inf_{w,z \in M; (1-a)\epsilon \le d(w,z) \le (1+a)\epsilon} \frac{t}{\int_0^t p(s,w,z) ds},   
\end{align} 
where in the fourth inequality we have used the assumption that $(\E, \F)$ is conservative.   

By \eqref{ass-growth-hk}, 
\[ \liminf_{t \to \infty} \inf_{w,z \in M; (1-a)\epsilon \le d(w,z) \le (1+a)\epsilon} \frac{f(t)}{\int_0^t p(s,w,z) ds} \ge \frac{1}{c_2}. \]

By this and \eqref{lower-mean}, we have assertion (ii). 
\end{proof}

\begin{proof}[Proof of Theorem \ref{pre-SLLN}]
We first show that $(\V)_t$ is a diffusion. 

\begin{Lem}\label{conti-vol}
For every $x \in M$, $\V$ is continuous with respect to $t$, $P^x$-a.s. 
\end{Lem}

\begin{proof}
We have that for every $\eta, t > 0$, 
\[ V_{\epsilon}(t+\eta) - V_{\epsilon}(t)  = \mu\left( \bigcup_{s \le t + \eta} B(X_s, \epsilon) \setminus \bigcup_{s \le t} B(X_s, \epsilon)  \right) \]
\[ \le \mu\left(B\left(X_t, \ \epsilon + \sup_{s \in [t, t+\eta]} d(X_t, X_s)\right) \setminus B\left(X_t, \epsilon\right)\right), \]
and, 
\[ V_{\epsilon}(t) - V_{\epsilon}(t-\eta)  = \mu \left( \bigcup_{s \le t} B(X_s, \epsilon) \setminus \bigcup_{s \le t - \eta} B(X_s, \epsilon) \right) \]
\[ \le \mu\left(B\left(X_t, \epsilon + \sup_{s \in [t-\eta, t]} d (X_t, X_s) \right) \setminus B\left(X_t, \epsilon - \sup_{s \in [t-3\eta, t-2\eta]} d(X_t, X_s) \right) \right), \]
where we adopt the notation that $B(x, r) := \emptyset$ for $r \le 0$.

Therefore,
\[ V_{\epsilon}(t+\eta) - V_{\epsilon}(t-\eta) \]
\[\le 2\mu\left(B\left(X_t, \ \epsilon + \sup_{s \in [t-\eta, t+\eta]} d (X_t, X_s)\right) \setminus B\left(X_t, \epsilon - \sup_{s \in [t-3\eta, t]} d(X_t, X_s)\right)\right). \]
Since $(X_t)_t$ is a diffusion, we have that $P^x$-a.s., for each $t > 0$, 
\[ \lim_{\eta \to 0} \sup_{s \in [t - 3\eta, t+\eta]} d (X_t, X_s) = 0. \]

Hence, we have that $P^x$-a.s., for every $t > 0$, 
\[ \limsup_{\eta \to 0} \mu\left(B\left(X_t, \ \epsilon + \sup_{s \in [t - \eta, t+\eta]} d (X_t, X_s)\right) \setminus B\left(X_t, \epsilon - \sup_{s \in [t-3\eta, t]} d(X_t, X_s) \right)\right) \]
\[\le  \mu\left(\overline{\partial}B\left(X_t, \epsilon\right)\right).  \]
By recalling \eqref{meas-vanish-bdry}, 
\[  \mu\left(\overline{\partial}B\left(X_t, \epsilon\right)\right) = 0.\] 
Hence, we have that $P^x$-a.s., for each $t > 0$, 
\[ \limsup_{\eta \to 0}  V_{\epsilon}(t+\eta) - V_{\epsilon}(t-\eta) \le 0. \]
By noting that $\V$ is non-decreasing with respect to $t$,  we have the assertion. 
\end{proof}

\begin{Rem}
We are not sure whether Lemma \ref{conti-vol} holds without \eqref{meas-vanish-bdry}. 
\end{Rem}

\begin{Prop}[{\cite[Theorem 3.1]{BK00}, \cite[Proposition A.2]{KKW17}}]\label{KKW-subadditive}
Let $X$ be a strong Markov process on $(M, d, \mu)$. 
Assume that $(H_t)_t$ is a continuous adapted non-decreasing functional of $X$ satisfying the following conditions.\\
(1) There exists an increasing function $\varphi$ on $(0,\infty)$ and a constant $c > 1$ such that 
\[ \lim_{b \to \infty} \sup_{x \in M, t > 0} P^x \left(H_t \ge b \varphi(t)\right) = 0,  \]
and, 
\[ \varphi(2r) \le c\varphi(r), \ \forall r > 0. \]
(2) \[ H_t \le H_s + H_{t-s} \circ \theta_s, \ t \ge s > 0. \]
Then, there exists a constant $C \in (0,\infty)$ such that for every $x \in M$, 
\[ \limsup_{t \to \infty} \frac{H_t}{\varphi(t/\log\log t) \log\log t} \le C, \textup{$P^x$-a.s.} \] 
\end{Prop} 

\begin{proof}
The proof is same as in the proof of \cite[Theorem 3.1]{BK00}, 
however, since we use a formula in the proof of \cite[Theorem 3.1]{BK00} below, we write down the details.   

We denote the largest integer of a real number $z$ which is less than or equal to $z$ by $[z]$. 
For $t > 0$ and $s \in [0,1]$, we let 
\[ B_{t,s} := \frac{H_{s t/[\log\log t]}}{\varphi(t/\log\log t)}. \]
By assumption (1), there exists a constant $b > 0$ such that 
\[ \sup_{t > 0, x \in M} P^x (B_{t,1} > b) \le \frac{1}{2}. \]
For $n \ge 1$, we let 
\[ T_n := \inf\left\{s \ge 0 : B_{t,s} \ge bn\right\}. \]
Then by the continuity of $H_t$, 
we have that $B_{T_n} = bn$ if $T_n < +\infty$, 
\[ P^x (B_{t,1} > b(n+1)) = P^x \left(B_{t,1} - B_{t, T_n} > b, T_n < 1\right). \]
By the strong Markov property of $X$ and the assumption that $H_t$ is adapted and non-decreasing, 
\begin{align*} 
P^x (B_{t,1} - B_{t, T_n} > b, T_n < 1) &\le E^x \left[P^{X_{T_n}} (B_{t,1} > b), T_n < 1\right] \\
&\le \frac{1}{2} \sup_{y \in M} P^y (B_{t,1} > bn). 
\end{align*} 

Hence, for each $n \ge 1$, 
\[ \sup_{t > 0, x \in M} P^x (B_{t,1} > bn) \le \frac{1}{2^n}.\]

Hence there exists a constant $a > 0$ such that 
\[ \sup_{x \in M, t > 0} E^x \left[\exp(a B_{t,1})\right] < +\infty. \]

We have that for every $x \in M$, 
\[ P^x \left(  \frac{H_t}{\varphi(t/\log\log t) [\log\log t]} \ge \lambda \right)  = P^x \left( \sum_{i = 1}^{[\log\log t]} (B_{t, i} - B_{t, i-1}) \ge  [\log\log t] \lambda \right)   \] 
\[ \le \exp\left(-a\lambda  [\log\log t] \right) E^x \left[ \prod_{i = 1}^{[\log\log t]} \exp\left( a (B_{t, i} - B_{t, i-1})  \right)\right]    \] 
\begin{equation}\label{supply} 
\le \left( \exp(-a\lambda) \sup_{y \in M, u > 0} E^y \left[\exp(a B_{u,1})\right] \right)^{[\log\log t]}, 
\end{equation}
where we have used the exponential Chebyshev's inequality in the first inequality and the Markov property of $(X_t)_t$ and assumption (2) in the second inequality. 

If we take a sufficiently large $\lambda_0 > 0$, then, there exists a constant $p > 1$ such that for every $x$ and every $t$, 
\[ P^x \left(  \frac{H_t}{\varphi(t/\log\log t) [\log\log t]} \ge \lambda_0 \right) \le \exp(-p [\log\log t]). \]
By this and the Borel-Cantelli lemma, we have that for every $x \in M$, 
\[ \limsup_{k \to \infty}  \frac{H_{\exp(k)}}{\varphi(\exp(k)/\log k) [\log k]} \le \lambda_0,  \textup{ $P^x$-a.s.}  \]
By this, the doubling property for $\varphi$ in assumption (1), and the assumption that $H_t$ is non-decreasing, we have the assertion. 
\end{proof}

Now we return to the proof of Theorem \ref{pre-SLLN}. 
By Lemma \ref{conti-vol}, $\{\V\}_{t \ge 0}$ is a diffusion.  
In the rest of this proof, we let $f$ be the function given by \eqref{def-f}. 
By applying Proposition \ref{KKW-subadditive} to the case that $H_t = \V$ and $\varphi(t) = t/f(t)$, 
it holds that there exists a positive non-random constant $C_0$ such that for every $x \in M$, 
\begin{equation}\label{upper-LIL-fund} 
\limsup_{t \to \infty} \frac{\V}{t/f(t/\log\log t)} \le C_0, \textup{ $P^x$-a.s.} 
\end{equation}
By Theorem \ref{tail-01} in Appendix, we have \eqref{non-str-rec-sup} for some non-negative constant $c_2$.    
 
We now show \eqref{non-str-rec-inf}.  
Thanks to $\textup{Vol}(V; \alpha_1, \alpha_2)$, we can apply Theorem \ref{318-KKW} in Appendix and we have that 
there exist two positive non-random constants $C_1$ and $C_2$ such that for every $x \in M$, 
\begin{align*}
 \liminf_{t \to \infty} \frac{\V}{V(\phi^{-1}(t/\log\log t))} &\le  \liminf_{t \to \infty} \frac{V\left(\epsilon + \sup_{s \le t} d(x, X_s) \right)}{V(\phi^{-1}(t/\log\log t))} \\
&\le  \liminf_{t \to \infty} 1+ \frac{C_1 \left(\epsilon^{\alpha_2} + \left(\sup_{s \le t} d(x, X_s)\right)^{\alpha_2}\right)}{\phi^{-1}(t/\log\log t)^{\alpha_2}} \\
&\le C_2 < +\infty,  \textup{ $P^x$-a.s.} 
\end{align*} 
By this, \eqref{upper-LIL-fund} and Theorem \ref{tail-01} in Appendix, we have \eqref{non-str-rec-inf}. 
Thus the proof of (i) is completed.

We show (ii). 
Let $\varphi(t) := t/f(t)$. 
Assume that $c_2 = 0$. 
Then, 
\begin{equation}\label{contradict-SLLN} 
\lim_{t \to \infty} \frac{\V}{\varphi(t/\log\log t) \log\log t} = 0, \textup{ $P^x$-a.s.}  
\end{equation}  

By Fubini's theorem and \eqref{supply}, we have that for  $\lambda_1 > 0$, 
\[ E^x \left[ \frac{\V}{\varphi(t/\log\log t) \log\log t}, \  \frac{\V}{\varphi(t/\log\log t) [\log\log t]}  \ge \lambda_1\right] \] 
\[= \int_{\lambda_1}^{\infty} P^x \left( \frac{\V}{\varphi(t/\log\log t) [\log\log t]}  \ge \lambda \right) d\lambda + \lambda_1 P^x \left( \frac{\V}{\varphi(t/\log\log t) [\log\log t]}  \ge \lambda_1 \right) \] 
\[ \le \int_{\lambda_1}^{\infty}  \left( \exp(-a\lambda) \sup_{y \in M, u > 0} E^y \left[\exp(a B_{u,1})\right] \right)^{[\log\log t]} d\lambda \]
\[+ \lambda_1 \left( \exp(-a\lambda_1) \sup_{y \in M, u > 0} E^y \left[\exp(a B_{u,1})\right] \right)^{[\log\log t]}\] 
\[ = (\frac{1}{a [\log\log t]} + \lambda_1) \left( \exp(-a\lambda_{1}) \sup_{y \in M, u > 0} E^y \left[\exp(a B_{u,1})\right] \right)^{[\log\log t]}. \] 
Hence we have that for sufficiently large $\lambda_1 > 0$, 
\begin{equation}\label{contradict-upper-small} 
\lim_{t \to \infty} E^x \left[ \frac{\V}{\varphi(t/\log\log t) \log\log t}, \  \frac{\V}{\varphi(t/\log\log t) [\log\log t]}  \ge \lambda_1\right] = 0. 
\end{equation} 
For every $\delta \in (0, \lambda_1)$, 
\[ E^x \left[ \frac{\V}{\varphi(t/\log\log t) \log\log t}, \  \frac{\V}{\varphi(t/\log\log t) [\log\log t]} < \lambda_1\right] \]
\[\le \delta + \lambda_1 P^x \left(  \frac{\V}{\varphi(t/\log\log t) [\log\log t]} \ge \delta \right).  \]
By this and \eqref{contradict-SLLN}, 
\[ \limsup_{t \to \infty} E^x \left[ \frac{\V}{\varphi(t/\log\log t) \log\log t}, \  \frac{\V}{\varphi(t/\log\log t) [\log\log t]} < \lambda_1\right] \le \delta. \] 
Since $\delta$ can be taken arbitrarily close to $0$, 
\[ \limsup_{t \to \infty} E^x \left[ \frac{\V}{\varphi(t/\log\log t) \log\log t}, \  \frac{\V}{\varphi(t/\log\log t) [\log\log t]} < \lambda_1\right] = 0. \] 
By this and \eqref{contradict-upper-small}, 
\[ \lim_{t \to \infty} \frac{E^x \left[ \V \right]}{t/f(t/\log\log t)} = \lim_{t \to \infty} \frac{E^x \left[ \V \right]}{\varphi(t/\log\log t) \log\log t} = 0. \] 
If \eqref{ratio} holds, then this convergence contradicts Theorem \ref{thm-fractal}. 
Hence $c_2 > 0$. 
Thus the proof of (ii) is completed.

We show (iii).  
Let 
\[ X[0,t] := \{X_s : s \in [0,t]\}, \ t > 0. \]
Let
\[ V_{0+}(t) := \mu\left(X[0,t]\right) = \lim_{\epsilon \to 0} \V. \]
This is the volume of the range of $X$. 

By Theorem \ref{411415} in Appendix, there exists a positive constant $c_1^{\prime}$ such that  for every $x \in M$, 
\[ \liminf_{t \to \infty} \frac{\mu(X[0,t])}{V(\phi^{-1}(t/\log\log t))} = c_1^{\prime}, \textup{ $P^x$-a.s.} \]
Hence $c_1 \ge c_1^{\prime} > 0$.  

Finally we show that $c_2 > 0$.
\[ c_2 = \limsup_{t \to \infty} \frac{\V}{V(\phi^{-1}(t/\log\log t))\log\log t} \ge  \limsup_{t \to \infty} \frac{\mu(X[0,t])}{V(\phi^{-1}(t/\log\log t))\log\log t}, \textup{ $P^x$-a.s.} \]
By  Theorem \ref{411415} in Appendix,  
there exists a positive constant $c_2^{\prime}$ such that  for every $x \in M$, 
\[ \limsup_{t \to \infty} \frac{\mu(X[0,t])}{V(\phi^{-1}(t/\log\log t)) \log\log t } = c_2^{\prime}, \textup{ $P^x$-a.s.} \]
Hence $c_2 > 0$. 
Thus the proof of Theorem \ref{pre-SLLN} is completed. 
\end{proof}

\subsection{Examples}

In this subsection we give examples to which Theorems \ref{thm-fractal} and \ref{pre-SLLN} are applicable. 

We first consider that case that $\textup{Vol}(\alpha)$ and FHK$(\beta)$ hold for some $\alpha, \beta > 1$.  
By Theorem \ref{pre-SLLN}, we have the following: 
\begin{Cor}
Assume that Ahlfors regularity $\textup{Vol}(\alpha)$ and full heat kernel estimate FHK$(\beta)$ hold. 
Let $f$ be the function defined as follows: 
\begin{equation*}\label{def-f-ab}
f(t)= \left\{
\begin{array}{l}
1 \ \ \ \ \ \ \ \ \ \ \alpha > \beta,\\
\log t  \ \ \ \ \ \ \alpha = \beta,\\
 t^{1- \alpha / \beta} \ \ \ \alpha < \beta.
\end{array}
\right.
\end{equation*}
Then,\\
(i) (transient or weakly recurrent cases) 
If $\alpha \ge \beta$, then, there exist two constants $c_1 \in [0, \infty)$ and $c_2 \in (0, \infty)$ depending on $\epsilon$ such that for every $x \in M$, 
\begin{equation*}\label{non-str-rec} 
c_1 = \liminf_{t \to \infty} \frac{\V}{t / f(t)} \le \limsup_{t \to \infty} \frac{\V}{t / f(t)} = c_2, \textup{ $P^x$-a.s.} 
\end{equation*} 
(ii) (strongly recurrent case) If $\alpha < \beta$, then, there exist two constants $c_3, c_4 \in (0, \infty)$ depending on $\epsilon$ such that for every $x \in M$, 
\begin{equation*}\label{str-rec-lower} 
\liminf_{t \to \infty} \frac{\V}{(t/\log\log t)^{\alpha/\beta}} = c_3, \textup{ $P^x$-a.s.,} 
\end{equation*} 
and, 
\begin{equation*}\label{str-rec-upper} 
\limsup_{t \to \infty} \frac{\V}{t^{\alpha/\beta} (\log\log t)^{1-\alpha/\beta}} = c_4, \textup{ $P^x$-a.s.} 
\end{equation*} 
\end{Cor}

For example, by Barlow-Bass \cite{BB92}, assertion (i) is applicable to Brownian motion on the $d$-dimensional standard Sierpinski carpets  for $d \ge 3$. 
By Barlow-Bass \cite{BB99}, it is also applicable to the standard Menger sponges for $d \ge 3$. 
By Barlow-Perkins \cite{BP88}, assertion (ii) is applicable to Brownian motion on the $d$-dimensional standard Sierpinski gaskets for  $d \ge 2$.

Theorems \ref{thm-fractal} and \ref{pre-SLLN} are also applicable to several models under assumptions weaker than $\textup{Vol}(\alpha)$ and FHK$(\beta)$. 
Barlow-Bass \cite{BB00} deals with the pre-carpet, which is a fractal-like manifold. 
This case corresponds to the case that $V(r) \simeq r^{d}$ and $\phi(r) \simeq r^{d_w}$ as $r \to 0$, and $V(r) \simeq r^{d_f}$ and $\phi(r) \simeq r^{2}$ as $r \to \infty$, where the constants $d_f$ and $d_{w}$ are called fractal dimension and walk dimension respectively. 
It holds that $d_f < d$ and $2 < d_w$.  
By \cite[Theorem 5.3]{BB00}, 
FHK$(\phi; 2, d_w)$ holds for a function $\phi$ satisfying \eqref{phi-beta}. 
Therefore, Theorem \ref{pre-SLLN} (i) is applicable to this case. 
We have that $f(t) \asymp t^{d_f / 2}$ if $d_f < 2$ and $f(t) \asymp 1$ if $d_f > 2$. 

If $d_f > 2$, then, 
\[ \frac{t}{f(t/\log\log t)} \asymp t, \ \ V(\phi^{-1}(t/\log\log t)) \asymp \left(\frac{t}{\log\log t}\right)^{d_f /2}, \ \ \  t \to \infty, \]
and hence, there exist two (non-random) constants $c_1$ and $c_2$  such that for every $x \in M$, 
\[ 0 \le c_1 = \liminf_{t \to \infty} \frac{\V}{t} \le  \limsup_{t \to \infty} \frac{\V}{t} = c_2 < +\infty, \textup{$P^x$-a.s.} \]
If $d_f = 2$, then, 
\[ \frac{t}{f(t/\log\log t)} \asymp \frac{t}{\log t}, \ \ V(\phi^{-1}(t/\log\log t)) \asymp \frac{t}{\log\log t}, \ \ \  t \to \infty, \]
and hence, there exist two (non-random) constants $c_1$ and $c_2$  such that for every $x \in M$, 
\[ 0 \le c_1 = \liminf_{t \to \infty} \frac{\V}{t/\log t} \le  \limsup_{t \to \infty} \frac{\V}{t/\log t} = c_2 < +\infty, \textup{$P^x$-a.s.} \]
If $d_f < 2$, then, 
\[ \frac{t}{f(t/\log\log t)} \asymp t^{d_f /2} (\log \log t)^{1 - d_f/2},  \ \ V(\phi^{-1}(t/\log\log t)) \asymp \left(\frac{t}{\log\log t}\right)^{d_f /2}, \ \ \  t \to \infty. \]
and hence, there exist two (non-random) constants $c_1$ and $c_2$  such that for every $x \in M$, 
\[  \liminf_{t \to \infty} \frac{\V}{t^{d_f /2} (\log\log t)^{-d_f /2}} = c_1, \ \textup{$P^x$-a.s.} \]
\[ \limsup_{t \to \infty} \frac{\V}{t^{d_f /2} (\log \log t)^{1 - d_f/2}} = c_2, \ \textup{$P^x$-a.s.}  \]
However, in each of the above cases, we are not sure whether $c_1$ or $c_2$ is positive or not.  

Grigor'yan and Saloff-Coste  \cite[Subsection 4.3]{GSC05} deals with the case that it {\it fails} that there exists a constant $c > 1$ such that $c^{-1}V(r) \le \mu(B(x,r)) \le cV(r)$ holds for every $x \in M$ and $r > 0$. 
In \cite[Subsection 4.3]{GSC05}, a class of radially symmetric weighted manifolds are considered.  
It is interesting to investigate their case, however, our method is not applicable to their case in direct manners. 
We also consider radially symmetric Riemannian manifolds in Section 4 below, however, our manifolds are different from those in \cite[Subsection 4.3]{GSC05}. 

\subsection{Remark about radial asymptotic}

We give a remark about radial asymptotic of $\V$ as $\epsilon \to +0$. 

\begin{Prop}[Radial asymptotic]
(i) If 
\begin{equation}\label{diverge-int} 
 \int_0^{1} \frac{ds}{V(\phi^{-1}(s))} = +\infty, 
\end{equation}  
then, for every $t > 0$ and every $x \in M$, $V_{0+}(t) = 0$, $P^x$-a.s.\\
(ii) If 
\begin{equation}\label{converge-int}
\int_0^{1} \frac{ds}{V(\phi^{-1}(s))} < +\infty, 
\end{equation} 
then, for every $t > 0$ and every $x \in M$, $V_{0+}(t) > 0$, $P^x$-a.s. 
\end{Prop}

In particular,  
if $\textup{Vol}(\alpha)$ and FHK$(\beta)$ hold and $\alpha \ge \beta$, then, $V_{0+}(t) = 0$, $P^x$-a.s., and if $\textup{Vol}(\alpha)$ and FHK$(\beta)$ hold and $\alpha < \beta$, then, $V_{0+}(t) > 0$, $P^x$-a.s. 

\begin{proof}
(i) Let $\epsilon > 0$. 
By HK$(\phi; \beta_1, \beta_2)$,  
\[ \int_0^t p(s,w,z) ds \ge \int_{\min\{t, c_6^{-1}\phi(d(w,z))\}}^{t} \frac{ds}{V(\phi^{-1}(s))}. \]
Hence, 
\begin{equation}\label{ineq-epsilon} 
\inf_{w,z \in M; d(w,z) \le 3\epsilon/2} \int_0^t p(s,w,z) ds \ge \int_{\min\{t, c_6^{-1}\phi(3\epsilon/2)\}}^{t} \frac{ds}{V(\phi^{-1}(s))}. 
\end{equation}   

By \eqref{phi-beta}, $\lim_{s \to 0+} \phi(0) = 0$. 
We recall that $\phi$ is an increasing function.  
By these properties and the monotone convergence theorem, 
\[ \lim_{\epsilon \to 0} \int_{\min\{t, c_6^{-1}\phi(3\epsilon/2)\}}^{t} \frac{ds}{V(\phi^{-1}(s))} = \int_{0}^{t} \frac{ds}{V(\phi^{-1}(s))}. \]
By this, \eqref{ineq-epsilon} and \eqref{diverge-int}, 
\[ \lim_{\epsilon \to 0}  \inf_{w,z \in M; d(w,z) \le 3\epsilon/2}  \int_0^t p(s,w,z) ds = \lim_{\epsilon \to 0} \int_{\min\{t, c_6^{-1}\phi(3\epsilon/2)\}}^{t} \frac{ds}{V(\phi^{-1}(s))} = +\infty. \]
By using this and applying \eqref{upper-mean} to the case that $a = 1/2$, 
it follows from the Lebesgue convergence theorem that for every $x \in M$, 
\[ \lim_{\epsilon \to 0+}  E^x [\V] = 0. \]
By this and the monotone convergence theorem, 
\[ E^x [V_{0+}(t)]  = 0. \]

(ii) In the same manner as in the proof of  \cite[Proposition 4.3]{KKW17}, 
we see that by \eqref{converge-int},  
the local time of $X$ exists, specifically, there exists a random field $\ell (x,t)(\omega)$ such that 
$\ell (x,t)(\omega)$ is jointly measurable with respect to $(t,x,\omega)$, and 
\[ \int_0^t h(X_s) ds = \int_{M} h(x) \ell (x,t) \mu(dx) \]
holds for every $T > 0$ and every Borel measurable function $h$ on $M$. 
Hence, 
\[ 0 < t = \int_{0}^{t} 1_{X[0,t]}(X_s) ds = \int_{X[0,t]} \ell (x,t) \mu(dx). \]
Hence, $\mu(X[0,t]) > 0$. 
\end{proof}


\section{Processes on bounded modifications}

Let $\M = (\R^d, d, \mu, X_t, P^x)$ be a bounded modification of the quintuple of $\R^d$, $d$, the Lebesgue measure and the standard Brownian motion. 
Throughout this section, $C_{\M}$ and $c_{\M}$ denote positive constants depending only on $\M$. 
If $\M$ is the quintuple of $\R^d$, $d$, the Lebesgue measure and the standard Brownian motion, then, we denote $C_{\M}$ and $c_{\M}$ by $C_{\textup{BM}}$ and $c_{\textup{BM}}$ respectively.
We denote by $p_{\M} (t,x,y)$ the heat kernel of $\M = (\R^d, d, \mu, X_t, P^x)$. 
For $y \in \R^d$ and $\epsilon_2 > \epsilon_1 > 0$, let 
\[ A\left(y, \epsilon_1, \epsilon_2\right) := \left\{z \in \R^d : \epsilon_1 \le d(z,y) \le \epsilon_2 \right\}. \]

We remark that we do {\it not} change the metric structure, in particular the metric is the Euclid metric and hence geodesic. 
Therefore we can apply \cite[Theorem 3.2]{BGK12} and we have that 
FHK$(2)$ is equivalent to the {\it parabolic Harnack inequality} (The precise definition of this inequality  is long in the framework of metric measure spaces, so we omit it here. See Barlow-Bass-Kumagai \cite[Remark 2.2]{BBK06} and \cite{BGK12} for details.), 
arguing as in the proof of \cite[Lemma 4.6]{BGK12},   
for every $z_1, z_2 \in A\left(y, \epsilon_1, \epsilon_2\right)$,   
\begin{equation}\label{conti-by-phi} 
\left| p_{\M} (s, y, z_1) - p_{\M} (s, y, z_2)\right| \le 1 \wedge C_{\M} \left(\frac{d (z_1, z_2)}{s}\right)^{2}. 
\end{equation} 

First we consider the case that $d \ge 3$. 
Let the {\it Green function}  
\[ G^{\M}(x,y) := \int_{0}^{\infty} p_{\M}(t,x,y) dt.  \]
 
\begin{proof}[Proof of Theorem \ref{thm-finite} (i)]
By FHK$(2)$, 
it holds that for every $T > 2$
\[ \left| \left(\int_0^T p_{\M}(t,x,y) dt \right)^{-1} -  \frac{1}{G^{\M}(x,y)} \right|  = \left| \frac{\int_0^T p_{\M}(t,x,y) dt  - \int_0^{\infty} p_{\M}(t,x,y) dt }{\int_0^T p_{\M}(t,x,y) dt \int_0^{\infty} p_{\M}(t,x,y) dt } \right| \]
\[ \le \frac{\int_T^{\infty} p_{\M}(t,x,y) dt }{\left(\int_0^T p_{\M}(t,x,y) dt\right)^2 }  \]
\[ \le \frac{\int_T^{\infty} c_{\M, 1} t^{-d/2} \exp(- c_{\M, 1} d(x,y)^2 / t) dt}{\left(\int_1^2 p_{\M}(t,x,y) dt\right)^2}  \]
\[ \le \frac{c_{\M, 1} T^{1-d/2}}{\left( c_{\M, 2} 2^{-d/2} \exp(- c_{\M, 2} d(x,y)^2) \right)^2},  \]
where $c_{\M, 1}$ and $c_{\M, 2}$ are positive constants independent from $x, y, t, T$.

Hence it holds that  for every $x, y \in \R^d$ satisfying that $(1-a)\epsilon \le d(x,y) \le (1+a)\epsilon$,  
\[ \left| \left(\int_0^T p_{\M}(t,x,y) dt \right)^{-1} -  \frac{1}{G^{\M}(x,y)} \right| \le \frac{c_1 T^{1-d/2}}{c_2^2 2^{-d} \exp(-2c_2 (1+\epsilon)^2 a^2)}. \]

Hence it holds that for every $a \in (0,1)$, 
\begin{equation*} 
\lim_{T \to \infty} \sup_{x, y \in \R^d; (1-a)\epsilon \le d(x,y) \le (1+a)\epsilon} \left| \left(\int_0^T p_{\M}(t,x,y) dt \right)^{-1} -  \frac{1}{G^{\M}(x,y)} \right| = 0. 
\end{equation*} 
Therefore, we can show the following by modifying the proof of  Lemma \ref{lem-fund-rough} a little.    

\begin{Lem}\label{tr-lem} 
For every $\epsilon > 0$, $a \in (0,1)$, $x \in \R^d$ , and for each integer $n \ge 1$, \\ 
(i) 
\[ \limsup_{t \to \infty} \frac{E^x [\V]}{t} \] 
\begin{equation*}
\le \frac{\sup_{w \in \R^d \setminus B(x,n)}  \mu(B(w, a \epsilon))}{\inf_{w \in \R^d \setminus B(x,n)}  \mu(B(x, a \epsilon))} \sup_{y,z \in \R^d \setminus B(x,n), (1-a)\epsilon \le d(y,z) \le (1+a)\epsilon} 
\frac{1}{G^{\M}(y,z)}. 
\end{equation*} 
(ii) 
\[ \liminf_{t \to \infty} \frac{E^x [\V]}{t} \]
\begin{equation*}
\ge \frac{\inf_{w \in \R^d \setminus B(x,n)} \mu(B(w, a \epsilon))}{\sup_{w \in \R^d \setminus B(x,n)} \mu(B(x, a \epsilon))} \inf_{y,z \in \R^d \setminus B(x,n), (1-a)\epsilon \le d(y,z) \le (1+a)\epsilon} \frac{1}{G^{\M}(y,z)}. 
\end{equation*}
\end{Lem}

Let $x \in \R^d$. 
By Lemma \ref{tr-lem} above, we have that for every $a \in (0,1)$ and large $n$, 
\[ \inf_{y, z \in \R^d \setminus B(0,n), (1-a)\epsilon \le d(y, z) \le (1+a)\epsilon} 
\frac{1}{G^{\M} (y,z)} \le \liminf_{t \to \infty} \frac{E_{\M}^x [\V]}{t} \] 
\begin{equation}\label{any-a}
\le \limsup_{t \to \infty} \frac{E_{\M}^x [\V]}{t} \le \sup_{y, z \in \R^d \setminus B(0,n), (1-a)\epsilon \le d(y, z) \le (1+a)\epsilon} 
\frac{1}{G^{\M} (y,z)}. 
\end{equation}

Since \eqref{any-a} holds for every $a \in (0,1)$,  
it suffices to show that for every $x \in \R^d$, 
\begin{equation}\label{g-g} 
\lim_{y; d(x,y) \to \infty} \sup_{z \in A\left(y, \epsilon_1, \epsilon_2\right)}  \left| G^{\M} (y, z) -  G^{\textup{BM}}(y, z) \right| = 0.
\end{equation}
Indeed, if this holds, then, 
\[ \inf_{y, z \in \R^d; (1-a)\epsilon \le d(y, z) \le (1+a)\epsilon} 
\frac{1}{G^{\textup{BM}} (y,z)} \le  \liminf_{t \to \infty} \frac{E_{\M}^x [\V]}{t} \]
\[\le \limsup_{t \to \infty} \frac{E_{\M}^x [\V]}{t} \le \sup_{y, z \in \R^d; (1-a)\epsilon \le d(y, z) \le (1+a)\epsilon} 
\frac{1}{G^{\textup{BM}} (y,z)}, \]
and the assertion follows if $a \to 0$. 

Let $ \mu_{\textup{BM}}$ be the Lebesgue measure on $\R^d$.  
Fix $y \in \R^d$. 
Let $u$ be a non-negative bounded continuous function supported on $A\left(y, \epsilon_1, \epsilon_2\right)$. 

Since $D$ is bounded and the point $x$ is fixed as a starting point of the process $(X_t)_t$, we have that if $d(x,y)$ is sufficiently large, then,  $B(y, d(x, y)/2) \cap D = \emptyset$. 
Indeed, if $B(y, d(x,y)/2) \cap D \ne \emptyset$, then, there exists a point $w \in D$ such that $d(w, y) \le d(x, y)/2$. 
Since $D$ is a bounded subset of $\R^d$, $\sup_{z \in D} d(x, z) < +\infty$. 
If $d(x, y)$ is sufficiently large, then, $\sup_{z \in D} d(x, z) < d(x,y)/3$, and hence, if $B(y, d(x,y)/2) \cap D \ne \emptyset$, then, 
$$d(x,y) \le d(x,w) + d(w, y) \le \sup_{z \in D} d(x, z) + \frac{d(x, y)}{2} \le  \frac{5d(x,y)}{6},$$ 
which is a contradiction. 
Hence $B(y, d(x,y)/2) \cap D = \emptyset$. 

By this and Definition \ref{bddm-def}, we have that  if $d(x,y)$ is sufficiently large, then, 
\[ E^{y}_{\M} \left[ u\left(X^{\M}_{t \wedge \tau_{B(y, d(x,y)/2)}}\right) \right] = E^{y}_{\textup{BM}} \left[ u\left(X^{\R^d}_{t \wedge \tau_{B(y, d(x,y)/2)}}\right) \right]. \]

Hence,  if $d(x,y)$ is sufficiently large, then,
\[ \int_{\R^d} u(z) G^{\M} (y,z) \mu_{\M} (dz) - \int_{\R^d} u(z)  G^{\textup{BM}}(y, z)  \mu_{\textup{BM}} (dz) \]
\[ = \int_0^{\infty} E^{y}_{\M} \left[u(X^{\M}_t)\right] - E^{y}_{\textup{BM}} \left[u(X^{\R^d}_t)\right]  dt \]
\[ = \int_0^{\infty} E^{y}_{\M} \left[u(X^{\M}_t) - u\left(X^{\M}_{t \wedge \tau_{B(y, d(x,y)/2)}}\right) \right] - E^{y}_{\textup{BM}} \left[u(X^{\R^d}_t) - u\left(X^{\R^d}_{t \wedge \tau_{B(y, d(x,y)/2)}}\right) \right] dt. \]

We have that for every $t_0 > 0$, 
\[ \left| \int_0^{\infty} E^{y}_{\M} \left[u\left(X^{\M}_t\right) - u\left(X^{\M}_{t \wedge \tau_{B(y, d(x,y)/2)}}\right) \right] dt \right| \]
\[= \left| \int_0^{\infty} E^{y}_{\M} \left[u\left(X^{\M}_t\right), \  t \ge \tau_{B(y, d(x,y)/2)} \right] dt \right| \]
\[ \le \|u\|_{\infty} \int_0^{\infty} P_{\M}^y \left(X^{\M}_t \in A\left(y, \epsilon_1, \epsilon_2\right), \, t \ge \tau_{B(y, d(x,y)/2)}\right) dt \] 
\[ \le \|u\|_{\infty} \int_0^{t_0} +  \int_{t_0}^{\infty} P_{\M}^y \left(X^{\M}_t \in A\left(y, \epsilon_1, \epsilon_2\right), \, t \ge \tau_{B(y, d(x,y)/2)}\right) dt \] 
\[ \le \|u\|_{\infty} \left(t_0 P_{\M}^y \left(t_0 \ge \tau_{B(y, d(x,y)/2)}\right) + C_{\M} t_{0}^{1-d/2}\right)\] 
\[ = \|u\|_{\infty} \left(t_0 P_{\R^d}^0 \left(t_0 \ge \tau_{B(0, d(x,y)/2)}\right) + C_{\M} t_{0}^{1-d/2}\right), \] 
where in the second inequality we have used $\textup{supp}(u) \subset A\left(y, \epsilon_1, \epsilon_2\right)$, in the fourth inequality we have used FHK$(2)$, and in the fifth equality we have used Definition \ref{bddm-def}.  
In the same manner we have the almost same estimate also for $\R^d$. 

Let $\mathcal{U}\left(A\left(y, \epsilon_1, \epsilon_2\right)\right)$ be the space of non-negative bounded continuous functions supported on $A\left(y, \epsilon_1, \epsilon_2\right)$ whose supremum norm is smaller than or equal to one. 
We have that for every $t_0 > 0$, 
\[ \sup_{u \in \mathcal{U}\left(A\left(y, \epsilon_1, \epsilon_2\right)\right)} \left|  \int_{\R^d} u(z) G^{\M} (y,z) \mu_{\M} (dz) - \int_{\R^d} u(z)  G^{\textup{BM}}(y, z)  \mu_{\textup{BM}} (dz) \right| \]
\[ \le \|u\|_{\infty} \left(2t_0 P_{\R^d}^0 \left(t_0 \ge \tau_{B(0, d(x,y)/2)}\right) + (C_{\M} + C_{\R^d})t_{0}^{1-d/2}\right), \]
where $C_{\R^d}$ is the constant corresponding to $C_{\M}$. 

Since $d \ge 3$ and $t_0$ can be taken arbitrarily large, we have that 
\begin{equation*}
\lim_{y; d(x,y) \to \infty} \sup_{u \in \mathcal{U}\left(A\left(y, \epsilon_1, \epsilon_2\right)\right)} \left|  \int_{\R^d} u(z) G^{\M} (y,z) \mu_{\M} (dz) - \int_{\R^d} u(z)  G^{\textup{BM}}(y, z)  \mu_{\textup{BM}} (dz) \right| = 0.  
\end{equation*} 

By \eqref{conti-by-phi}, it holds that  for $z_1, z_2 \in A\left(y, \epsilon_1, \epsilon_2\right)$, 
\[ \left|G^{\M} (y, z_1) - G^{\M} (y, z_2)\right| \le \int_{0}^{d(z_1, z_2)} 1 dt + \int_{d(z_1, z_2)}^{\infty} C_{\M} \left(\frac{d (z_1, z_2)}{t}\right)^{2} dt \]
\begin{equation}\label{phi} 
\le (C_{\M} + 1) d(z_1, z_2).      
\end{equation} 
We have the almost same estimate also for $\R^d$. 

Recall that if $B(x,r) \cap D = \emptyset$, then, by Definition 1.3, $\mu(B(x,r))$ equals the volume of a ball of radius $r$ in $\R^d$, in particular, $\mu(B(x,r))  \asymp r^d$. 
If \eqref{g-g} fails, then, by \eqref{phi}, 
we can construct an $L^{\infty}$-bounded function $u$ which is supported on $A\left(y, \epsilon_1, \epsilon_2\right)$ and  
satisfies that 
\[ \limsup_{y; d(x,y) \to \infty} \left|  \int_{\R^d} u(z) G^{\M} (y,z) \mu_{\M} (dz) - \int_{\R^d} u(z)  G^{\textup{BM}}(y, z)  \mu_{\textup{BM}} (dz) \right| > 0. \]
Now \eqref{g-g} follows. 
\end{proof}


Second we consider the case that $d = 2$. 
For $x \in \R^2$ and $A \subset \R^2$,  let 
\[ d(x, A) := \inf_{y \in A} d(x,y). \]
For $E \subset \R^2$ and $\epsilon > 0$,  let 
\[ E_{\epsilon} := \left\{x \in \R^d : d(x, \R^d \setminus E) \ge \epsilon \right\}. \]
It is obvious that $E_{\epsilon} \subset E$ holds for every $\epsilon > 0$. 

\begin{proof}[Proof of Theorem \ref{thm-finite} (ii)]  
Recall the argument in the proof of Lemma \ref{lem-fund-rough}. 
It follows from Lemma \ref{lem-fund} that for every Borel measurable $E \subset \R^2$, $\eta = a\epsilon$, $a \in (0,1)$ and $T, t > 0$,  
\[ \int_{E} P^x (T_{B(y, \epsilon)} \le t) \mu(dy) \le  \int_{E}\frac{ \int_{0}^{t+T} P^x \left(X_{s} \in B(y, a\epsilon)\right) ds}{\inf_{w \in \overline{\partial}B(y, \epsilon)}  \int_{0}^{T} P^w (X_{s}  \in B(y, a\epsilon )) ds} \mu(dy) \]
\[ \le  \frac{\int_{E} \int_{0}^{t+T} P^x \left(X_{s} \in B(y, a\epsilon )\right) ds \mu(dy)}{\inf_{y \in E} \mu(B(y, a\epsilon)) \inf_{z, w \in \R^2; (1-a)\epsilon \le d(z,w) \le (1+a)\epsilon}  \int_{0}^{T} p(s,w,z) ds} \]
\begin{equation}\label{upper-2d-fund} 
\le  \frac{\sup_{y \in \R^2} \mu\left(B(y, a\epsilon ) \cap E\right) (t+T)}{\inf_{y \in E} \mu(B(y, a\epsilon)) \inf_{z, w \in \R^2; d(w, E) \le \epsilon, (1-a)\epsilon \le d(z,w) \le (1+a)\epsilon}  \int_{0}^{T} p(s,w,z) ds}, 
\end{equation}
and we also have that 
\[ \int_{E} P^x (T_{B(y, \epsilon)} \le t) \mu(dy) \ge \int_{E} \frac{\int_0^t P^x \left(X_s \in B(y, a\epsilon)\right) ds}{\sup_{w \in \partial B(y, \epsilon)} \int_0^t P^w (X_s \in B(y, a\epsilon)) ds} \mu(dy) \]added
\[ = \int_{E} \frac{\int_0^t \int_{B(y, a\epsilon)} p(s,x,z) \mu(dz) ds}{\sup_{w \in \partial B(y, \epsilon)} \int_0^t \int_{B(y, a\epsilon)} p(s,w,z) \mu(dz) ds} \mu(dy) \]
\[ \ge \frac{\int_{E} \int_0^t \int_{B(y, a\epsilon)} p(s,x,z) \mu(dz) ds \mu(dy) }{\sup_{y \in E} \left\{\mu(B(y, a\epsilon)) \sup_{w \in \partial B(y, \epsilon)}  \sup_{z \in B(y, a\epsilon)} \int_0^t  p(s,w,z) ds \right\}} \]
\[ = \frac{\int_{E} \int_0^t \mu(B(z, a\epsilon)) p(s,x,z) \mu(dz) ds  }{\sup_{y \in E} \left\{\mu(B(y, a\epsilon)) \sup_{w \in \partial B(y, \epsilon)}  \sup_{z \in B(y, a\epsilon)} \int_0^t  p(s,w,z) ds \right\}} \]
\begin{equation}\label{lower-2d-fund}  
\ge  \frac{\inf_{z \in E_{\epsilon}} \mu(B(z, a\epsilon) ) \int_0^t \int_{E_{\epsilon}} p(s,x,z) \mu(dz) ds }{\sup_{y \in E} \mu(B(y, a\epsilon)) \sup_{d(w, E) \le \epsilon, (1-a)\epsilon \le d(z,w) \le (1+a)\epsilon}  \int_{0}^{t} p(s,w,z) ds}. 
\end{equation} 

Let $x \in \R^2$. 
Let $m$ be a large integer. 
Let 
\begin{equation}\label{def-Fm} 
F_m (t) := \frac{t^{m/(m+1)}}{(\log t)^3}. 
\end{equation}    
Let 
\[ S_x \left(t, \epsilon_1, \epsilon_2\right) := \left\{(y,z) \in (\R^d)^2 : d(x,y) + \epsilon_2 \ge t^{1/2} / \log t, z \in A\left(y, \epsilon_1, \epsilon_2\right) \right\}. \] 

We now show that 
\begin{Lem}\label{Sx} 
For every $\epsilon_2 > \epsilon_1 > 0$, 
\[ \lim_{t \to \infty} \sup_{(y,z) \in S_x \left(t, \epsilon_1, \epsilon_2\right)} 
\left| \left(\int_{0}^{ F_m (t) } p_{\M} (s,y,z) ds\right)^{-1} - \left(\int_{0}^{ F_m (t) } p_{\textup{BM}} (s,y,z) ds\right)^{-1} \right| \log t = 0. \]
\end{Lem}

\begin{proof}
Thanks to FHK$(2)$, it suffices to show that 
\begin{equation}\label{g-g-2} 
\lim_{t \to \infty} \sup_{(y,z) \in S_x (t, \epsilon_1, \epsilon_2) } \left| \int_{0}^{ F_m (t) } p_{\M} (s,y,z) - p_{\textup{BM}}(s, y, z) ds \right| \frac{1}{\log t} = 0.
\end{equation}

Fix $y \in \R^2$.  
Let $u$ be a non-negative bounded continuous function supported on $A\left(y, \epsilon_1, \epsilon_2\right)$. 
Since $A\left(y, \epsilon_1, \epsilon_2\right) \cap D = \emptyset$, 
$u$ is well-defined also on $\R^2$. 
In the same manner as in the case that $d \ge 3$, 
by using the Burkholder-Davis-Gundy inequality for degree $m \ge 1$, 
\[ \left| \int_{0}^{F_m (t) } \left(\int_{\R^2} u(z) p_{\M} (s,y,z) \mu_{\M} (dz) - \int_{\R^2} u(z) p_{\textup{BM}} (s, y, z) \mu_{\textup{BM}} (dz) \right) ds\right| \]
\[ \le \| u \|_{\infty} \int_{0}^{F_m (t) } P_{\textup{BM}}^0 \left(s \ge T_{\R^2 \setminus B(0, t^{1/2} / \log t)}\right) ds \]
\[ \le 2\| u \|_{\infty} \int_{0}^{F_m (t) } P_{\textup{BM}}^0 \left(\max_{0 \le s^{\prime} \le s} \left| X_{s^{\prime}} \right| \ge  t^{1/2} / 2\log t \right) ds \]
\[ \le C_m \int_{0}^{F_m (t) } s^m ds \frac{1}{t^{m} / (\log t)^{2m}} \le \frac{C_m \| u \|_{\infty}}{(\log t)^{m+3}}.  \]
In the third display above, $(X_t)_t$ denotes the one-dimensional standard Brownian motion and $P_{\textup{BM}}^0$ denotes the law of the one-dimensional standard Brownian motion starting at the origin.  

By \eqref{conti-by-phi},  
in the same manner as in the derivation of \eqref{phi}, 
we see that 
\[ \sup_{z_1, z_2 \in A\left(y, \epsilon_1, \epsilon_2\right)} \int_{0}^{F_m (t) } \left| p_{\M} (s, y, z_1) - p_{\M} (s, y, z_2) \right| ds\] 
\[ \le \sup_{z_1, z_2 \in A\left(y, \epsilon_1, \epsilon_2\right)} \int_{0}^{\infty} \left| p_{\M} (s, y, z_1) - p_{\M} (s, y, z_2) \right| ds \] 
\begin{equation}\label{0-Fm} 
\le \sup_{z_1, z_2 \in A\left(y, \epsilon_1, \epsilon_2\right)} (1+C_{\M}) d(z_1, z_2) \le 2(1+C_{\M}) \epsilon_2, 
\end{equation} 
where $C_{\M}$ is the same constant appearing in \eqref{conti-by-phi}.  
Furthermore, by FHK$(2)$,  
\begin{align}\label{Fm-t}  
\int_{F_m (t)}^{t} p_{\M} (s, y, z) ds &\le c_{\M} \int_{F_m (t)}^{t} s^{-1} ds \notag\\
&= c_{\M} \log \frac{t}{F_m (t)} = c_{\M} \left(\frac{\log t}{m+1} + 3 \log\log t\right). 
\end{align} 
The same estimate holds also for the standard Brownian motion, that is, 
\[ \int_{F_m (t)}^{t} p_{\textup{BM}} (s, y, z) ds \le c_{\textup{BM}}  \left(\frac{\log t}{m+1} + 3 \log\log t\right).  \]

By using this, \eqref{0-Fm} and \eqref{Fm-t}, 
it holds that  there exists a constant $C$ such that 
\[ \limsup_{t \to \infty} \sup_{(y, z) \in S_x (t, \epsilon_1, \epsilon_2) } \left| \int_{0}^{t} p_{\M} (s,y,z) - p_{\textup{BM}}(s, y, z) ds \right| \frac{1}{\log t} \le \frac{C}{m+1} \] 
holds for every $m$, and hence, \eqref{g-g-2} follows. 
\end{proof}

Now we show the upper bound. 
We first remark that 
\[ \lim_{t \to \infty} \frac{\log t}{t} \mu_{\M} \left( B(x,  t^{1/2}/\log t) \right) = 0. \]
By this and \eqref{base}, 
\begin{equation}\label{reduced} 
\limsup_{t \to \infty} \frac{\log t}{t} E_{\M}^x \left[\V\right] = \limsup_{t \to \infty} \frac{\log t}{t} \int_{E_{(t)}} P_{\M}^x (T_{B(y, \epsilon)} \le t) \mu_{\M}(dy), 
\end{equation}
where we let $E_{(t)} := \R^2 \setminus B(x,  t^{1/2}/\log t)$ for $t > 0$.  

By applying \eqref{upper-2d-fund} to the case that $E = E_{(t)}$, $a = 1/2$ and $T = F_m (t)$, 
\[ \int_{E_{(t)}} P_{\M}^x (T_{B(y, \epsilon)} \le t) \mu_{\M}(dy) \]
\[\le  \frac{\sup_{y \in \R^2} \mu_{\M}\left( E_{(t)}  \cap B(y, \epsilon/2)\right) (t+F_m (t))}{\inf_{y \in E_{(t)} } \mu_{\M}(B(y, \epsilon/2)) \inf_{z, w; d(w, E_{(t)}) \le \epsilon, \epsilon/2 \le d(z,w) \le 3\epsilon/2}  \int_{0}^{F_m (t)} p(s,w,z) ds}\]

Since $E_{(t)}  \cap D = \emptyset$ for any large $t$, it holds that for every large $t$, 
\[ \sup_{y \in \R^2} \mu_{\M}\left( E_{(t)} \cap B(y, \epsilon/2)\right) = \inf_{y \in E_{(t)} } \mu_{\M}(B(y, \epsilon/2)). \]
Hence, it holds that for every large $t$, 
\[ \int_{E_{(t)}} P_{\M}^x (T_{B(y, \epsilon)} \le t) \mu_{\M}(dy) \le \frac{t +  F_m (t) }{\inf_{(y,z) \in S_x (t, \epsilon/2, 3\epsilon/2)}  \int_{0}^{ F_m (t) } p_{\M} (s,y,z) ds}.  \]

By applying this, \eqref{reduced} and Lemma \ref{Sx},  
it holds that  for every $m \ge 1$, 
\[ \limsup_{t \to \infty} \frac{\log t}{t} E_{\M}^x \left[\V\right] = \limsup_{t \to \infty} \frac{\log t}{t} \int_{\R^d \setminus B(x,  t^{1/2}/\log t)} P_{\M}^x (T_{B(y, \epsilon)} \le t) \mu_{\M}(dy) \]
\[\le \limsup_{t \to \infty} \frac{t +  F_m (t) }{t} \frac{\log t}{\inf_{(y,z) \in S_x (t, \epsilon/2, 3\epsilon/2)}  \int_{0}^{ F_m (t) } p_{\M} (s,y,z) ds} \]
\[ \le \limsup_{t \to \infty} \left(1+\frac{F_m (t) }{t} \right) \frac{\log t}{\inf_{\epsilon/2 \le d(y,z) \le 3\epsilon/2}  \int_{0}^{ F_m (t) } p_{\textup{BM}}(s,y,z) ds}. \]

By FHK(2), we have that for each fixed $\epsilon > 0$, 
\[ \lim_{t \to \infty} \frac{1}{\log F_m (t)} \int_{0}^{ F_m (t) } p_{\textup{BM}}(s,y,z) ds =  \frac{1}{2\pi}. \]
This convergence is uniform with respect to $y$ and $z$ satisfying that $\epsilon/2 \le d(y,z) \le 3\epsilon/2$.  
By this and the definition of $F_m (t)$ in \eqref{def-Fm},  we have that 
\[ \lim_{t \to \infty} \left(1+\frac{F_m (t) }{t} \right) \frac{\log t}{\inf_{\epsilon/2 \le d(y,z) \le 3\epsilon/2}  \int_{0}^{ F_m (t) } p_{\textup{BM}}(s,y,z) ds} = \frac{m+1}{m} 2\pi. \]

Since $m$ can be taken arbitrarily largely, we have that 
\[ \limsup_{t \to \infty} \frac{\log t}{t} E_{\M}^x \left[\V\right]  \le 2\pi. \]  

Now we show the lower bound. 
By FHK$(2)$ and conservativeness, it is easy to see that   
\[ \lim_{t \to \infty} \frac{1}{t}  \int_0^t \int_{E_{(t)}}  p_{\M}(s, x, z) \mu(dz) ds = 1.  \]

By applying  this, Lemma \ref{Sx} and \eqref{lower-2d-fund} to the case that $E = E_{(t)}$ and $a = 1/2$, 
\[ \liminf_{t \to \infty} \frac{\log t}{t} E_{\M}^x \left[\V\right] \]
\[ \ge \liminf_{t \to \infty} \frac{1}{t} \int_0^t \int_{E_{(t)}}  p_{\M} (s, x, z) \mu(dz) ds \frac{\log t}{\sup_{(w, z) \in S_x (t, \epsilon/2, 3\epsilon/2)} \int_{0}^{t} p_{\M} (s,w,z) ds} \]
\[ =  \liminf_{t \to \infty}  \frac{\log t}{\sup_{(w, z) \in S_x (t, \epsilon/2, 3\epsilon/2)} \int_{0}^{t} p_{\textup{BM}} (s,w,z) ds} = 2\pi. \]
Thus we have the assertion.   
\end{proof}

\begin{Rem}
If $d \ge 3$, 
$G(x,y)$ depends on the value of $d(x,y)$. 
If $d = 2$,  then, for every $x,y \in \R^2$, 
\[ p(t,x,y) \sim 2\pi \log t, \ t \to \infty.\]   
In this sense, in the case that $d = 2$, it is not important how to take $\epsilon_1$ and $\epsilon_2$.  
\end{Rem} 


\section{Fluctuation results}

Before we state the proof, we prepare notation. 

\begin{Def}[Diffusion on Riemannian manifolds]\label{Rmfd-def}
Let $\M = (M, g)$ be a connected Riemannian manifold and $\mu_{\M}$ be the Riemannian volume.  
Let $\nabla_{\M} f $ be the weak gradient of $f$.  
Let 
\[ W^1 := \left\{f \in L^2 (M, \mu_{\M}) : \nabla_{\M} f \in L^2 (M, \mu_{\M}) \right\}. \] 
For $f, g \in W^1$, let 
\[ \E_{\M} (f,g) := \int_M \left(\nabla_{\M} f, \nabla_{\M} g\right) d\mu_{\M}.  \]
Here $(,)$ is the canonical inner product on $\R^d$. 
Let $\F_{\M} $ be the closure of $C^{\infty}(M) \textup{ under } W^1$.  
Then, $(\E_{\M}, \F_{\M})$ is a strongly-local regular Dirichlet form in $L^2 (M, \mu_{\M})$.  
Let $\left((X^{\M}_t)_{t \ge 0}, (P_{\M}^x)_{x \in M}\right)$ be an associated diffusion.   
\end{Def}

If we consider a Riemannian manifold $M$, then we may simply write $\M = (M, g, \E_{\M}, \F_{\M})$.  
Let $p_{\M} (t,x,y)$ be the minimal positive fundamental solution of the heat equation on $M$. 
Then, for every $h \in L^2 (M, \mu_{\M})$, every $t \ge 0$ and {\it every} $x \in M$,   
\[ \int_M h(y) p_{\M} (t,x,y) \mu_{\M}(dy) = E_{\M}^x \left[h(X^{\M}_t)\right]. \]

In the following, we deal with the case that $M = \R^d$ but $g$ is not the Euclid metric. 
However we consider a class of metric measure spaces which consist of the Euclid space $\R^d$, the Euclid distance $d$, the Riemannian volume $\mu_{\M}$, and $(\E_{\M}, \F_{\M})$ for a Riemannian manifold $\M$. 
We emphasize that we consider the Euclid distance instead of the Riemannian distance $d_{\M}$ defined by the Riemannian metric of $\M = (\R^d, g)$. 
Informally speaking we do not consider ``singular" $\M$ in this section. 
We consider only the case that  the Riemannian distance $d_{\M}$ of $\M$ is equivalent to the Euclid distance, and hence the topologies of $\R^d$ induced by these two distances are identical with each other. 
In this case $(\E_{\M}, \F_{\M})$ and an associated diffusion $\left((X^{\M}_t)_{t \ge 0}, (P_{\M}^x)_{x \in M}\right)$ are well-defined. 
We can consider a metric measure Dirichlet space $\left(\R^d, d, \mu_{\M}, \E_{\M}, \F_{\M}\right)$ and an associated diffusion of it.

Before we proceed to the proof, we give a rough sketch of the proof. 
The arguments below are somewhat informal. 
\begin{proof}[Outline of Proof of Theorem \ref{thm-fluc-1}]
We first construct  a specific sequence of metric measure Dirichlet spaces $(\M_k = (\R^d, d, \mu_{\M_k}, \E_{\M_k}, \F_{\M_k}))_{k \ge 1}$ satisfying the following conditions:

(i) (Definitions \ref{Rmfd-def}, \ref{Rk-def}, and  \ref{Riem-def}) 
Roughly speaking, the local structure of $\M_k$ and $\M_{k+1}$ are identical with each other. 
There exists an increasing sequence $(R_k)_k$ such that $\mu_{M_{k+1}} = \mu_{M_k}$ on $B(0, R_k)$ and the restriction of $(\E_{\M_{k+1}}, \F_{\M_{k+1}}))$ to $B(0, R_k)$ is identical with that of $(\E_{\M_k}, \F_{\M_k}))$ to $B(0, R_k)$. 

(ii) (Lemma \ref{fm-check}) 
If $k$ is odd, then, a metric measure Dirichlet space $(\R^d, d, \mu_{\M_k}, \E_{\M_k}, \F_{\M_k})$ is a bounded modification of a metric measure Dirichlet space  $(\R^d, d, \mu_{A}, \E_{A}, \F_{A})$. \\
If $k$ is even, then, a metric measure Dirichlet space $(\R^d, d, \mu_{\M_k}, \E_{\M_k}, \F_{\M_k})$ is a bounded modification of a metric measure Dirichlet space $(\R^d, d, \mu_{B}, \E_{B}, \F_{B})$. 

(iii) (Lemma \ref{conv-thm-finite}) There exist two constants $0 < c(\mathcal{A}, \epsilon) < c(\mathcal{B}, \epsilon) < +\infty$ such that 
\[ \lim_{t \to \infty} \frac{E^{0}_{\M_{2k-1}}[\V]}{t} = c(\mathcal{A}, \epsilon), \ \textup{ and }  \]
\[ \lim_{t \to \infty} \frac{E^{0}_{\M_{2k}}[\V]}{t} = c(\mathcal{B}, \epsilon). \]

Conditions (ii) and (iii) above are independent from how to choose $(R_k)_k$.   

Then we consider a specific metric measure Dirichlet space $\M_{\infty} = (\R^d, d, \mu_{\M_\infty}, \E_{\M_\infty}, \F_{\M_\infty})$ as a limit of $(\M_k)_{k \ge 1}$. 
By choosing $R_{k+1}$ is much larger than $R_k$ for each $k$, 
we have a rapidly increasing sequence $(R_k)_k$. 
This is done between Lemma \ref{peculiar-k} and Proposition \ref{stability}. 
Then we have that 
\[  \liminf_{t \to \infty} \frac{E^{0}_{\M_{\infty}}\left[ \V\right]}{t} \le c(\mathcal{A}, \epsilon) < c(\mathcal{B}, \epsilon) \le \limsup_{t \to \infty} \frac{E^{0}_{\M_{\infty}}\left[ \V \right]}{t} \] 

Each $\M_k$ arises from the Brownian motion on a Riemannian manifold. 
In the definition of $\M_k$, we do {\it not} use the Riemannian distance and use the Euclid distance instead. 
However in the proof we use the Riemannian distance associated with $\mu_{\M_k}$ and it makes the arguments clearer. 

Our proof is somewhat analogous to the proof of \cite[Theorem 1.3]{O14}, which deal with the corresponding fluctuation result for the range of random walk on infinite graphs.
However, the continuous framework which we consider here is more technically involved. 
The technical difficulty arises from the fact that the diffusion process which we consider can be arbitrarily far from its starting point in arbitrarily small time.  
\end{proof}

Now we proceed to the proof. 

\begin{proof}[Proof of Theorem \ref{thm-fluc-1}]

\begin{Def}\label{Rk-def} 
(i) For an infinite sequence 
$0 = R_0 < R_1 < R_2 < R_3 < \cdots$ 
and for $k \ge 0$, we let 
$A_k := [R_{2k}, R_{2k+1} - 1)$,  
$C_{k,1} := [R_{2k+1} -1, R_{2k+1})$,  
$B_k := [R_{2k+1}, R_{2k+2}-1)$, and 
$C_{k,2} := [R_{2k+2}-1, R_{2k+2})$.\\
(ii) Let $k \ge 0$. 
Let $G_{2k}$ be a smooth non-negative function on $[0, +\infty)$ 
such that\\ 
(a) $G_{2k} = 1$ on $[R_{2k}, \infty)$\\
(b) for every $j < k$,  
$G_{2k} = 1$ on $A_j$, $G_{2k} = 4$ on $B_j$, and $1 \le G_{2k} \le 4$ on $C_{j,1} \cup C_{j,2}$.  \\
Let $G_{2k+1}$ be a smooth non-negative function on $[0, +\infty)$ 
such that\\ 
(a) $G_{2k+1} = 4$ on $[R_{2k+1}, \infty)$\\
(b) for every $j < k$,  
$G_{2k+1} = 1$ on $A_j$, $G_{2k+1} = 4$ on $B_j$, and $1 \le G_{2k+1} \le 4$ on $C_{j,1} \cup C_{j,2}$.    
Moreover, 
$G_{2k+1} = 1$ on $A_k$, and $1 \le G_{2k+1} \le 4$ on $C_{k,1}$. 
\end{Def} 

Then we have 
\[ G_{\infty} := \lim_{k \to \infty} G_k\] exists, and 
is a smooth non-negative function on $[0, +\infty)$ such that for every $j \ge 0$,  
$G_{\infty} = 1$ on $A_j$, $G_{\infty} = 4$ on $B_j$, and $1 \le G_{\infty} \le 4$ on $C_{j,1} \cup C_{j,2}$.  

We now define a sequence of radially symmetric Riemannian manifolds. 

\begin{Def}\label{Riem-def}
(i) Let a family of Riemannian manifolds $\M_k := (\mathbb{R}^d, g_k)$, $0 \le k \le \infty$, as follows: 
\[ g_k (x) := G_k (d(0, x)) I_d.  \]
Here $I_d$ is the $d \times d$ identity matrix.\\
(ii) Let $\mathcal{A} = (\R^d, g_A)$ and $\mathcal{B} = (\R^d, g_B)$ be Riemannian manifolds such that their Riemannian metrics are given by $g_A \equiv I_d$ and $g_B \equiv 4I_d$ on each point of $\R^d$.  
\end{Def} 
 
We remark that $\M_0$ is the Euclid space. 
$\M_k$ depends only on the choice of the sequence $R_1 < \cdots < R_k$.

\begin{figure}[H]
\centering
\includegraphics[width = 8cm, height = 6cm, bb = 0 0 800 600]{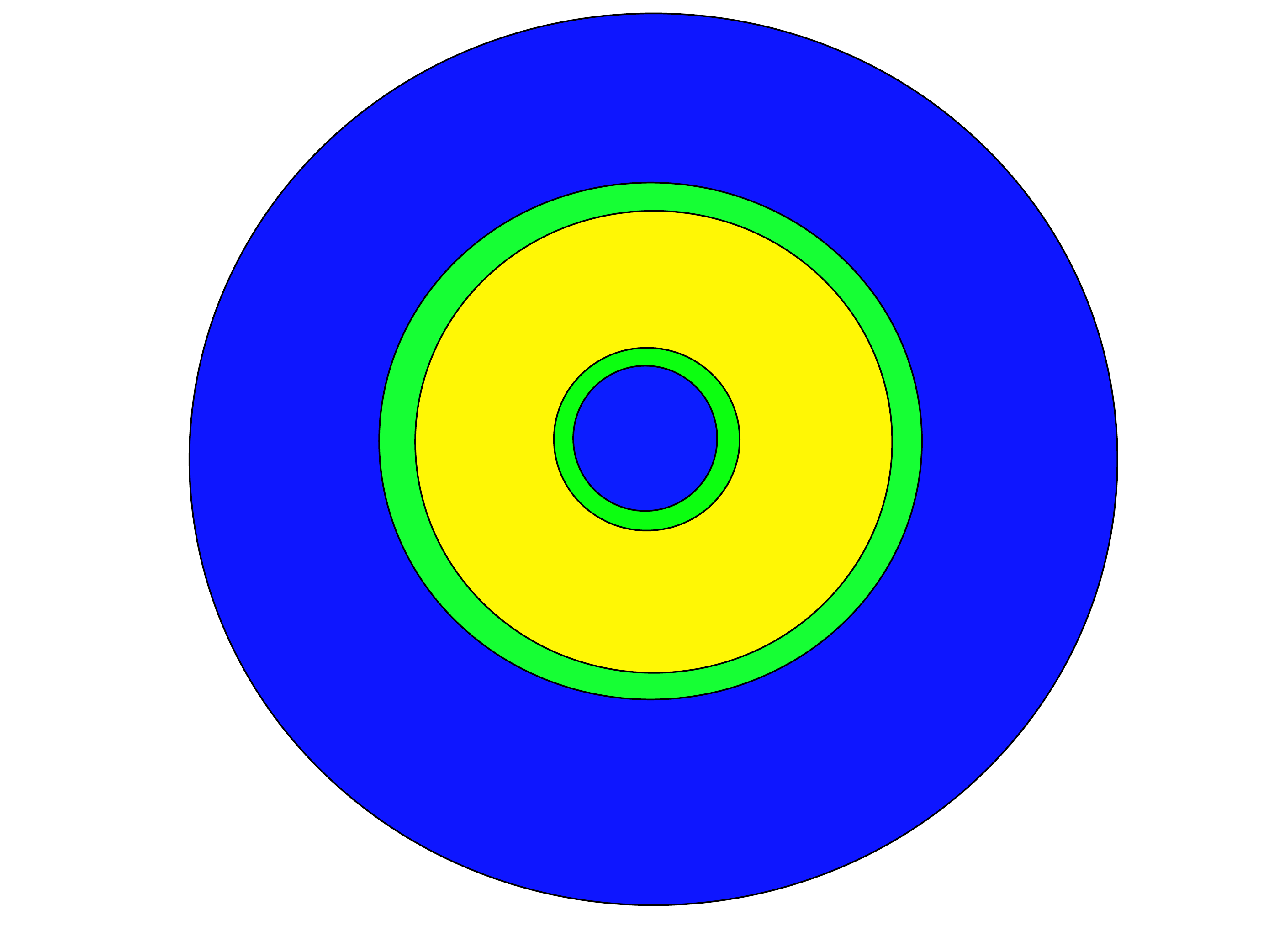}
\caption{\small Rough figure of $\M_k$. Riemannian metrics are the $d \times d$ identity matrix $I_d$ on the blue parts and are $2 \times I_d$ on the yellow part, 
The green parts are ``connecting" parts in order to assure smoothness and local regularity of $\M_k$.}
\end{figure}

For $x \in \R^d$, $t, \epsilon > 0$, 
let $B^{\M_j}_{\textup{Riem}}(x,\epsilon)$ be open balls with center $x$ and radius $\epsilon$ with respect to the Riemannian distance of $\M_j$.  
Since $1 \le G_k \le 4$ for every $k$, 
we can show that for every $0 \le j \le +\infty$,  
\begin{equation}\label{dist-equi} 
d(x,y) \le d^{\M_j}_{\textup{Riem}}(x,y) \le 2d(x,y). 
\end{equation} 
This does not depend on the choices of $\{R_k\}_k$. 

By \eqref{dist-equi}, each $\M_j$ is complete with respect to $d_{\textup{Riem}}$. 
By the Hopf-Rinow theorem, each $\M_j$ is geodesically complete. 

By \eqref{dist-equi}, 
there exists a positive constant $C_d$ such that for every $r > 0$, $x \in \R^d$ and every $j \ge 0$, 
\[ \mu_{\M_j}\left(B^{\M_j}_{\textup{Riem}}(x,r)\right) \le \mu_{\M_j}\left(B(x,r)\right) \le C_d (2r)^d. \]

By this and Grigor'yan \cite[Theorem 9.1]{G99}, 
we have the stochastic completeness, that is, 
for each $t > 0$, every $x \in \R^d$  and every $j \ge 0$, 
\[ \int_{\R^d} p_{\M_j} (t,x,y) \mu_{\M_j}(dy) = 1.\] 

Recall Definition \ref{bddm-def}. 
\begin{Lem}\label{fm-check}
If $k$ is odd, then, a metric measure space $(\R^d, d, \mu_{\M_k}, \E_{\M_k}, \F_{\M_k})$ is a bounded modification of a metric measure space  $(\R^d, d, \mu_{A}, \E_{A}, \F_{A})$. 
If $k$ is even, then, a metric measure space $(\R^d, d, \mu_{\M_k}, \E_{\M_k}, \F_{\M_k})$ is a bounded modification of a metric measure space $(\R^d, d, \mu_{B}, \E_{B}, \F_{B})$. 
\end{Lem}

\begin{proof}
We give a proof in the case that $k$ is odd.  
(M1) is obvious. 
(M2) follows from the definitions of the Dirichlet forms of $\M_{2k+1}$ and $\mathcal{A}$, and Shigekawa-Taniguchi \cite[Lemma 6.3]{ST92}.  
In the case that $k$ is even, we can show the assertion in the same manner.  
\end{proof}

\begin{Lem}\label{Riem-loc-reg}
For each $k$, $(\R^d, d, \mu_{\M_k}, \E_{\M_k}, \F_{\M_k})$ satisfies \textrm{FHK$(2)$}.  
\end{Lem}

\begin{proof}
By the definition of Riemannian metric $g$ on $\M_j$, $|g|$, $|\partial_i g|$ and $|\partial^{2}_{ik} g|$, $1 \le i, k \le d$, are  uniformly bounded on $\M_j$. 
Therefore, the Ricci curvature is bounded below. 
Therefore, by the Li-Yau estimates (\cite{LY86}), 
$\M_k$ satisfies a local version of the Parabolic Harnack inequality with parameter $2$. 
(See Barlow \cite{B04} for the definition of the Parabolic Harnack inequality with parameter $2$.) 
By the definition, as a Riemannian manifold, $\M_k$ is rough-isometric to $\R^d$ with the Euclid metric $d$. 
By using Hebisch and Saloff-Coste \cite[Theorem 2.7]{HSC01} and \cite[Theorem 5.4]{B04}, $(\M_k, \E)$ also satisfies the Parabolic Harnack inequality with parameter $2$, which is equivalent to FHK$(2)$. 
Now we replace the Riemannian metric with the Euclid metric $d$ and use \eqref{dist-equi}.   
\end{proof} 

\begin{Lem}[convergence results]\label{conv-thm-finite}
For $\M = \mathcal{A} \textup{ or } \mathcal{B}$, let 
\[ c(\M, \epsilon) := \inf_{x, y \in \R^d,  d(x,y) \le \epsilon} \frac{1}{ G^{\M} (x,y)} =\inf_{x, y \in \R^d, d(x,y) \le \epsilon} \frac{1}{\int_0^{\infty} p_{\M}(t,x,y) dt}.\]
Then we have \\
(i) \[ \lim_{t \to \infty} \frac{E^{0}_{\M_{2k-1}}[\V]}{t} = c(\mathcal{A}, \epsilon).  \]
(ii) \[ \lim_{t \to \infty} \frac{E^{0}_{\M_{2k}}[\V]}{t} = c(\mathcal{B}, \epsilon). \]
These convergences do not depend on choices of $\{R_k\}_k$.  
\end{Lem}

\begin{proof}
(i) follows from Lemmas \ref{fm-check}, \ref{Riem-loc-reg}, and Theorem \ref{thm-finite}. 
Now we show (ii). 
Let $\widetilde{\mathcal{B}} = (\R^d, d_{\textup{Riem}}^{\mathcal{B}}, \mu_{B},  \E_{B}, \F_{B})$ and 
let $(X_t, P^x)$ be a diffusion associated with $\widetilde{\mathcal{B}}$.  
Then the law is identical with the standard Brownian motion. 
Hence, 
\[ \lim_{t \to \infty} \frac{E^{0}_{\widetilde{\mathcal{B}}}[\V]}{t} = c(\mathcal{A}, \epsilon). \]

Let $\widetilde{\M_{2k}} = (\R^d, d_{\textup{Riem}}^{\mathcal{B}}, \mu_{\M_{2k}}, \E_{\M_{2k}}, \F_{\M_{2k}})$. 
Then, $\widetilde{\M_{2k}}$ is a bounded modification of $\widetilde{\mathcal{B}}$. 
By Theorem \ref{thm-finite}, 
\[ \lim_{t \to \infty} \frac{E^{0}_{\widetilde \M_{2k}}[\V]}{t} = c(\mathcal{A}, \epsilon). \]

Since 
\[ E^{0}_{\widetilde{\M_{2k}}}[V_{2\epsilon}(t)] = E^{0}_{\M_{2k}}[\V], \]
it holds that 
\begin{equation}\label{2epsilon} 
\lim_{t \to \infty} \frac{E^{0}_{\M_{2k}}[\V]}{t} = c(\mathcal{A}, 2\epsilon) = 2^{(d-2)/2}c(\mathcal{A}, \epsilon). 
\end{equation} 

Since the two Riemannian manifolds $\mathcal{A}$ and $\mathcal{B}$ are obtained by multiplying positive constants to the Euclid metrics, and $X^{\mathcal{A}}$ and $X^{\mathcal{B}}$ are Brownian motion on $\mathcal{A}$ and $\mathcal{B}$ respectively, 
\[ p_{\M}(t,x,y) = \left(\frac{1}{2\pi t}\right)^{d/2} \exp\left( - \frac{d^{\M}_{\textup{Riem}}(x,y)^2}{2t} \right),  \]
for $t > 0$ and $x, y \in \R^d$, $\M = \mathcal{A} \textup{ or } \mathcal{B}$.  
Since 
\[ d^{\mathcal{A}}_{\textup{Riem}}(x,y) = d(x,y) = \frac{d^{\mathcal{B}}_{\textup{Riem}}(x,y)}{2}, \] 
we have that 
\[ p_{\mathcal{A}}(t,x,y) =  \left(\frac{1}{2\pi t}\right)^{d/2} \exp\left( - \frac{d(x,y)^2}{2t} \right) \]
and 
\[ p_{\mathcal{B}}(t,x,y) =  \left(\frac{1}{2\pi t}\right)^{d/2} \exp\left( - \frac{2d(x,y)^2}{t} \right) \]
By integrating these quantities with respect to $t$, we have that  
\[ G^{\mathcal{A}} (x,y) = 2^{(d-2)/2} G^{\mathcal{B}} (x,y).\]  

Thus we have 
\[ 2^{(d-2)/2} c(\mathcal{A}, \epsilon) = c(\mathcal{B}, \epsilon).\] 
This and \eqref{2epsilon} complete the proof of assertion (ii). 
\end{proof}

We also have that 
\begin{Lem}[Uniform upper Gaussian heat kernel estimates]\label{uhk}
There exist two constants $c_1$ and $c_2$ such that for every $0 \le j \le +\infty$ and $x, y \in \R^d$, 
\[ p_{\M_j}(t,x,y) \le \frac{c_1}{t^{d/2}} \exp\left(-c_2 \frac{d^{\M_j}_{\textup{Riem}}(x,y)^2}{t}\right).  \]
\end{Lem} 

\begin{proof}
We first remark that by the Nash inequality \cite{N58}, there is a constant $c$ such that for every $f \in L^1 (\R^d, \mu_{\mathcal{A}}) \cap W^{1,2}(\R^d)$, 
\[ \left(\int_{\R^d} |f|^2 d\mu_{\mathcal{A}}\right)^{1 + 2/d} \le c \int_{\R^d} |\nabla f|^2 d\mu_{\mathcal{A}} \left(\int_{\R^d} |f| d\mu_{\mathcal{A}}\right)^{4/d}.  \]
where we let $W^{1,2}(\R^d)$ be the Sobolev space consisting of all real-valued functions on $\R^d$ whose weak derivatives are in $L^2 (\R^d, \mu_{\mathcal{A}})$. 
 
By the definition of $\M_j$, 
we have that for every Borel measurable subset $B$ of $\R^d$ and every $0 \le j \le +\infty$, 
\[ \mu_{\M_j}(B) \le \mu_{\mathcal{A}}(B) \le 2^d \mu_{\M_j}(B).  \]
Therefore, there is a general constant $c$ such that for every $0 \le j \le +\infty$ and every $f \in L^1 (\R^d, \mu_{\M_j}) \cap W(\R^d)$, 
\[ \left(\int_{\R^d} |f|^2 d\mu_{\M_j}\right)^{1 + 2/d} \le c \int_{\R^d} |\nabla f|^2 d\mu_{\M_j} \left(\int_{\R^d} |f| d\mu_{\M_j}\right)^{4/d}. \]

Therefore, by the Carlen-Kusuoka-Stroock \cite{CKS},  we have that 
there is a general constant $c$ such that for every $j$ and $x \in \R^d$,  
\[ \sup_{t > 0} t^{d/2} p_{\M_j}(t,x,x) \le c. \]  

The assertion now follows from this on-diagonal heat kernel upper bound and  Grigor'yan \cite[Theorems 3.1 and 3.2]{G97}.   
\end{proof} 

\begin{Lem}\label{Lem-1}
For every $t > 0$, 
\[ \lim_{R \to \infty} \sup_{0 \le j \le +\infty} \int_{\R^d \setminus B\left(0, R\right)} P^{0}_{\M_j} \left(  T_{B(x, \epsilon)} \le t  \right) d\mu_{\M_j}(x) = 0.  \]
\end{Lem} 

\begin{proof}
By the definition of $\mu_{\M_j}$ and Lemma \ref{uhk} above, 
there exists a constant $C_0$ such that for every $j$ and every $x \in \R^d$ and every $r, t > 0$, 
\[ \mu_{\M_j}(B(x, 2r)) \le C_0 \mu_{\M_j}(B(x, r)). \]
and
\[ p_{\M_j}(t,x,x) \le \frac{C_0}{\mu_{\M_j}(B(x, t^{1/2}))}. \]

We recall that $\M_j$ has a structure of Riemannian manifold and is complete with respect to the structure.  
Hence we can apply Grigor'yan and Saloff-Coste  \cite[Proposition 4.4]{GSC02} to this setting and we obtain that 
there exist constants $C_1$ and $C_2$ such that  for every $j$ and every $x \in \R^d \setminus B(0, 2\epsilon)$, 
\[ P^{0}_{\M_j} \left(  T_{B(x, \epsilon)} \le t  \right) \le C_1 \exp\left(-\frac{d^{\M_j}_{\textup{Riem}}(0, B(x, \epsilon))^2}{2t}\right) \le C_2  \exp\left(-\frac{d(0, B(x, \epsilon))^2}{2 C_2 t}\right). \]

Hence, for each fixed $t > 0$, 
\[ \sup_{0 \le j \le +\infty} \int_{\R^d \setminus B\left(0, R\right)} P^{0}_{\M_j} \left(  T_{B(x, \epsilon)} \le t  \right) d\mu_{\M_j}(x) \]
\[\le 2^d C_2 \int_{\R^d \setminus B\left(0, R\right)} \exp\left(-\frac{d(0, B(x, \epsilon))^2}{2 C_2 t}\right) \mu_{\mathcal A}(dx), \to 0, \ R \to +\infty. \]
\end{proof}

\begin{Lem}\label{Lem-2}
We have that for each $t > 0$, 
\[ \lim_{R \to \infty} R^d \sup_{0 \le j \le +\infty} P^{0}_{\M_j}\left( \tau_{B(0, R)} \le t  \right) = 0. \]
\end{Lem}

\begin{proof}
In this proof, we regard each $\M_j$ as a Riemannian manifold. 
Let us define two annuli \[ K_R := \overline{B}(0, R+1) \setminus B(0, R) \]
\[\subset O_R := B(0, R+d) \setminus \overline{B}(0, R-d). \]

We follow \cite[(1.9)]{GSC02} for the definition of $\ca_{\M_j}(K_R, O_R)$, specifically,  
\[ \ca_{\M_j}(K_R, O_R) \]
\[:= \inf\left\{\int_{O_R} \left|\nabla_{\M_j} \phi\right| d\mu_{\M_j} : \phi = 1 \textup{ on } K_R, \ \phi \textup{ has a compact support  in } O_R \right\}, \]
where the gradient $\nabla_{\M_j} \phi$ is taken {\it with respect to the Riemannian metric on $\M_j$}, but it differs from the corresponding gradient taken {\it with respect to the Euclid metric on $\R^d$} only by positive constants which are independent from $j$. 
Therefore, there exists a constant $C_d$ such that for every $j$ and every $\phi$ such that $\phi = 1$  on $K_R$ and $\phi$ has a compact support in $O_R$, 
\[ \int_{O_R} \left|\nabla_{\M_j} \phi\right| d\mu_{\M_j} \le C_d \int_{O_R} \left|\nabla_{\mathcal A} \phi\right| d\mu_{\mathcal A}. \]

Therefore we have that 
\[ \ca_{\M_j}(K_R, O_R) \le C_d  \ca_{\mathcal A}(K_R, O_R) \le C_d \mu_{\mathcal A}(O_R) d^{\mathcal A}_{\text{Riem}}\left(K_R, \R^d \setminus O_R\right)^2 \]
\[ \le C_{d, 1} \mu_{\M_j}(O_R) d^{\M_j}_{\text{Riem}}\left(K_R, \R^d \setminus O_R\right)^2, \]
and hence, 
\[ \sup_{0 \le j \le +\infty} \ca_{\M_j}(K_R, O_R) \le C_{d,2} R^d. \]
In the above two displays, $C_{d,i}, i = 1,2$, are constants depending only on $d$, and each of the constants does not depend on $0 \le j \le +\infty$ or $R > 0$. 

Since $K_R$ and $O_R$ are annuli defined with respect to the Euclid distance, 
we can show that 
\[ \ca_{\M_j}(K_R, O_R) > 0 \]   
in the same manner as in the proof of 
\[ \ca(K_R, O_R) > 0, \] 
where the capacity is defined on the Euclid space equipped with the Euclid distance and the Lebesgue measure. 
Therefore we can apply  \cite[Theorem 3.7]{GSC02} to this case and it holds that  
\[ P^{0}_{\M_j}\left( \tau_{B(0, R)} \le t  \right) \le \ca_{\M_j}(K_R, O_R) \int_0^t \sup_{y \in O_R \setminus K_R} p_{\M_j, B(0, R)}(s,0,y) ds. \]

Now it suffices to give a uniform upper bound for $p_{\M_j, B(0, R)}(s,0,y)$, $y \in O_R \setminus K_R$.
By Lemma \ref{uhk} (see also \cite[Remark 3.8]{GSC02}), we have that   
\[ p_{\M_j, B(0, R)}(s,0,y) \le p_{\M_j}(s,0,y) \le \frac{C_d}{s^{d/2}} \exp\left(-\frac{(R-1)^2}{2s}\right).  \]  
This leads to the assertion. 
\end{proof} 

\begin{Lem}\label{LY}
For each fixed $k \ge 1$ and each $t > 0$, 
there is a constant $R_{\M_k}(t)$ such that the following two inequalities hold:\\
(i) 
\[ E_{\M_k}^0 \left[V_{\epsilon}\left(t \wedge \tau_{B\left(0, R_{\M_k}(t)\right)}\right)  \right] \ge (1 - \exp(-t)) E_{\M_k}^0 \left[ \V \right].  \]\\
(ii) 
\[ \sup_{j}  E_{\M_j}^0 \left[ \V, \ t > \tau_{B(0, R_{\M_k}(t))}\right] \le 2^{-k}. \]
\end{Lem}

\begin{proof}
Fix $k \ge 1$ and $t > 0$. 
Since 
\[ \lim_{R \to \infty} P^0_{\M_k}\left(\tau_{B(0, R)} < t \right) = 0, \] 
we have that for each fixed $t > 0$ and $k \ge 1$, 
\[ \lim_{R \to \infty} E^0_{\M_k}\left[\V, \tau_{B(0, R)} < t \right] = 0. \]

Hence it holds that if we take $R = R_{\M_k} (t)$ sufficiently large for each fixed $t > 0$ and $k \ge 1$, then, 
\begin{align*} 
E^0_{\M_k} \left[V_{\epsilon}(t \wedge \tau_{B(0, R_{\M_k}(t))}) \right] &\ge E_{\M_k}^0 \left[\V, \tau_{B(0, R_{\M_k}(t))} \ge t \right] \\
&\ge (1 - \exp(-t)) E^0_{\M_k} \left[ \V \right]. 
\end{align*} 
Thus (i) holds if we take $R_{\M_k} (t)$ sufficiently large.

Furthermore, it holds that for every $R > 0$, 
\begin{align*} 
E_{\M_j}^0 \left[ \V, \ t \ge \tau_{B(0, R)}\right] 
&= \int_{\R^d} P^{0}_{\M_j}\left(T_{B(x, \epsilon)} \vee \tau_{B(0, R)} \le t \right) d\mu_{\M_j}(x)  \\
&= \int_{B(0, \epsilon + R)}  + \int_{\R^d \setminus B(0, \epsilon + R)} P^{0}_{\M_j}\left(T_{B(x, \epsilon)} \vee \tau_{B(0, R)} \le t \right) d\mu_{\M_j}(x)  \\
&\le \mu_{\M_j}(B(0, \epsilon + R)) P^{0}_{\M_j}\left( \tau_{B(0, R)} \le t  \right)   \\
&\ \ + \int_{\R^d \setminus B(0, \epsilon + R)} P^{0}_{\M_j} \left(  T_{B(x, \epsilon)} \le t  \right) d\mu_{\M_j}(x). 
\end{align*} 

By Lemmas \ref{Lem-1} and \ref{Lem-2}, 
we have that if $R_{\M_k}(t)$ is sufficiently large, then, 
\[  \sup_{j}  \mu_{\M_j}(B(0, \epsilon + R)) P^{0}_{\M_j}\left( \tau_{B(0, R)} \le t  \right) \le 2^{-k-1}, \]
and
\[  \sup_{j}  \int_{\R^d \setminus B(0, \epsilon + R)} P^{0}_{\M_j} \left(  T_{B(x, \epsilon)} \le t  \right) d\mu_{\M_j}(x) \le 2^{-k-1} \]
Hence, if $R_{\M_k}(t)$ is sufficiently large, then, 
\[ \sup_{j} E_{\M_j}^0 \left[ \V, \ t \ge \tau_{B(0, R_{\M_k}(t))}\right] \le 2^{-k}. \] 
Thus  (ii) holds if we take $R_{\M_k} (t)$ sufficiently large.
\end{proof}

\begin{Lem}[Specifying $t_k$]\label{peculiar-k}
Let $R_{\M_{k}}(t)$ be as in the above lemma. Then,\\
(i) If $k$ is even, then,  we can take $t_k$ such that 
\[ \frac{E_{\M_k}^0 \left[ V_{\epsilon}\left(t_k \wedge \tau_{B(0, R_{\M_{k}}(t_k))}\right) \right]}{t_k} \ge (1 - \exp(-k)) \left(c(\mathcal{B}, \epsilon) - 2^{-k}\right). \]
(ii) If $k$ is odd, then,  we can take $t_k$  such that 
\[ \frac{E_{\M_k}^0 \left[ V_{\epsilon}(t_k) \right]}{t_k} \le c(\mathcal{A}, \epsilon) + 2^{-k}. \] 
\end{Lem}

\begin{proof}
(i) Let $k$ be an even integer. 
By Lemma \ref{LY}, we have that for every $t > 0$, 
\[ \frac{E_{\M_k}^0 \left[ V_{\epsilon}\left(t \wedge \tau_{B(0, R_{\M_{k}}(t))}\right) \right]}{t} \ge (1 - \exp(-k)) \frac{E_{\M_k}^0 \left[ \V \right]}{t}. \]

By Lemma \ref{conv-thm-finite}, we have that for every sufficiently large $t > 0$, 
\[  \frac{E_{\M_k}^0 \left[ \V \right]}{t} \ge c(\mathcal{B}, \epsilon) - 2^{-k}. \]

Assertion (i) follows from these two results. 

(ii) By Lemma \ref{conv-thm-finite}, we have that for every sufficiently large $t > 0$, 
\[  \frac{E_{\M_k}^0 \left[ \V \right]}{t} \le c(\mathcal{A}, \epsilon) + 2^{-k}. \]
\end{proof} 

For each non-negative integer $k$, we now let $R_{k} := R_{\M_{k}}(t_{k})$ where $t_k$ is the constant appearing in Lemma \ref{peculiar-k}.
Thus $\{\M_j\}_j$ are explicitly defined for every $0 \le j \le +\infty$. 

\begin{Prop}[stability]\label{stability}  
We have that for every non-negative integer $k$, 
\[ \frac{E^0_{\M_{j}}\left[V_{\epsilon}(t_{2k-1})\right]}{t_{2k-1}} \le c(\mathcal{A}, \epsilon) + 2^{-(2k-1)}, \ \forall j \ge 2k-1. \] 
\[ \frac{E^0_{\M_{j}}\left[V_{\epsilon}(t_{2k})\right]}{t_{2k}} \ge c(\mathcal{B}, \epsilon) - 2^{-2k}, \ \forall j \ge 2k. \] 
We remark that $j$ can take $+\infty$. 
\end{Prop}

\begin{proof}
Due to the definition of $\M_j$, 
if we let $X^{(k)}_s := X_{s \wedge \tau_{B(0, R_k)}}$, 
then the law of the process $X^{(k)}$ under $P^0_{\M_j}$ is identical with each other for every $j \ge k$.
Hence we have that for every $j \ge k$, 
\begin{equation}\label{local-det} 
 E^0_{\M_j} \left[ V_{\epsilon}(t_k \wedge \tau_{B(0, R_k)})\right] 
= E^0_{\M_j} \left[\mu\left(\bigcup_{0 \le s \le t_k} B(X^{(k)}_s, \epsilon) \right)\right] = E^0_{\M_k} \left[ V_{\epsilon}(t_k \wedge \tau_{B(0, R_k)})\right]. 
\end{equation}

Assume that $k$ is odd. 
By Lemma \ref{LY},  \eqref{local-det} and  Lemma \ref{peculiar-k}, 
\begin{align*} 
E^{0}_{\M_{j}}[V_{\epsilon}(t_k)] &= E^{0}_{\M_{j}}\left[V_{\epsilon}(t_k), t_{k} > \tau_{B(0, R_{k})}\right] + E^{0}_{\M_{j}}\left[V_{\epsilon}(t_k), t_{k} \le \tau_{B(0, R_{k})}\right].  \\
&\le 2^{-k} + E^{0}_{\M_{j}}\left[V_{\epsilon}\left(t_{k} \wedge \tau_{B(0, R_{k})}\right) \right] \\
&= 2^{-k} + E^{0}_{\M_{k}}\left[V_{\epsilon}\left(t_{k} \wedge \tau_{B(0, R_{k})}\right) \right] \\
&\le 2^{-k} + E^{0}_{\M_{k}}\left[V_{\epsilon}(t_{k})\right] \\
&\le 2^{-k}(1 + t_k) + t_k c(\mathcal{A}, \epsilon). 
\end{align*} 

Assume that $k$ is even. 
By \eqref{local-det} and Lemmas \ref{peculiar-k} and \ref{LY},   
\begin{align*} 
E^{0}_{\M_{j}}\left[V_{\epsilon}(t_k)\right] &\ge E^{0}_{\M_{j}}\left[V_{\epsilon}(t_{k}), t_{k} \le \tau_{B(0, R_{k})}\right] \\
&= E^{0}_{\M_{k}}\left[V_{\epsilon}\left(t_{k} \wedge \tau_{B(0, R_{k})}\right), t_{k} \le \tau_{B(0, R_{k})}\right] \\
&= E^{0}_{\M_{k}}\left[V_{\epsilon}\left(t_{k} \right), t_{k} \le \tau_{B(0, R_{k})}\right] \\
&=  E^{0}_{\M_{k}}\left[V_{\epsilon}(t_{k} \wedge \tau_{B(0, R_{k})})\right] -  E^{0}_{\M_{k}}\left[V_{\epsilon}(t_k), t_{k} > \tau_{B(0, R_{k})}\right] \\
&\ge t_k (1 - \exp(-k)) (c(\mathcal{B}, \epsilon) - 2^{-k}) - 2^{-k}. 
\end{align*}
 
\end{proof}

Hence, 
\[ \liminf_{t \to \infty} \frac{E^{0}_{\M_{\infty}}\left[ \V\right]}{t} \le c(\mathcal{A}, \epsilon) < c(\mathcal{B}, \epsilon) \le \limsup_{t \to \infty} \frac{E^{0}_{\M_{\infty}}\left[ \V \right]}{t}. \]

Thus the proof of Theorem \ref{thm-fluc-1} is completed. 
\end{proof}

\begin{Rem}
We are not sure whether the above proof is applicable to the two-dimensional case with small modifications. 
\end{Rem}


\section{Further results for processes on bounded modifications}

Before we proceed to the proof of Theorem \ref{Lv2-mod}, we prepare a lemma for the standard Brownian motion on the Euclid space.  

\begin{Lem}\label{Lv2-lem-1}
Let $\widetilde\epsilon, R_0  > 0$. 
Then, 
\begin{equation}\label{return-from-remote} 
\limsup_{t \to \infty} t^{1+\widetilde\epsilon} \ \sup_{x \in \R^d \setminus B(0,  t^{(1+ \widetilde\epsilon)/(d-2)})} P^{x}_{\textup{BM}} \left(T_{B(0, R_0)} \le t\right) < +\infty. 
\end{equation} 
\end{Lem}

\begin{proof}
By applying Lemma \ref{lem-fund} (i) to the case that $\eta = \epsilon/2$ and $T = 1$, we have that  
\begin{align*} 
P^{x}_{\textup{BM}} (T_{B(0, R_0)} \le t) 
&\le \frac{\int_{0}^{t+1} P^x (X_s \in B(0, R_0/2)) ds}{\inf_{w \in \overline{\partial}B(0, R_0)} \int_{0}^{1} P^w (X_s \in B(0, R_0/2)) ds} \\
&\le C\int_0^t s^{- d/2} \exp\left(-\frac{(|x| - R_0 /2)^2}{s}\right) ds, 
\end{align*}  
where in the last inequality we have used FHK(2) for $\R^d$.  

By changing of variable $u = (|x| - R_0 /2)/s^{1/2}$,
then, we have that 
$$ds = -2 (|x| - R_0 /2)^2 u^{-3} du$$ and $$s^{- d/2} = u^d (|x| - R_0 /2)^{-d}.$$ 

It holds that 
\[ \int_0^t s^{- d/2} \exp\left(-\frac{(|x| - R_0 /2)^2}{s}\right) ds \le 2 \int_{0}^{\infty} u^{d-3} \exp(-u^2) du \left(|x| - R_0 /2\right)^{2-d}.  \]
By recalling $|x| \ge  t^{(1+ \widetilde\epsilon)/(d-2)}$, we have the assertion.    
\end{proof}

\begin{proof}[Proof of Theorem \ref{Lv2-mod}]  
Let $d \ge 6$. 
Hereafter, for ease of notation, $\left| W \right| $ denotes the volume of a random Borel subset $W$ of $\R^d$ under $\mu$ or the Lebesgue measure.  
If $W$ is a random subset under $\M$, then we choose $\mu$, and, if $W$ is a random subset with respect to the standard Brownian motion, then we choose the Lebesgue measure, 
Assume that $D \subset B(0, R_0)$.    
Let 
\[ W_{s, t} := \bigcup_{u \in [s, t]} B(X_u, \epsilon), \ 0 \le s \le t. \]
We remark that by \eqref{upper-mean} and FHK(2), 
\begin{align}\label{0-1-finite} 
\sup_{y \in \R^d} E^y_{\M} \left[ \left| W_{0,1} \right| \right] &= \sup_{y \in \R^d} E^y_{\M} \left[ V_{\epsilon}(1) \right] \notag\\
 &\le C \frac{1}{\inf_{\epsilon/2\le d(z,w) \le 3\epsilon/2} \int_0^1 p(s,z,w) ds} < +\infty. 
\end{align} 

Let 
\[ g(t) := t^{(1+ \widetilde\epsilon)/(d-2)} \] 
and 
\[ \widetilde D(t) :=  B(0, g(t)), \ t > 0.  \]

By the assumption, 
there exists a constant $C$ such that $\mu_{\M} (B(0,r)) \le C r^d$ holds for every $r > 0$. 
Using this, $d \ge 6$, FHK$(2)$, and \eqref{return-from-remote},  
it holds that  
if we choose sufficiently small $\widetilde\epsilon > 0$ in Lemma \ref{Lv2-lem-1},   
then, for some $a > 0$,   
\begin{equation}\label{M-exit}  
P_{\M}^0 \left(\tau_{\widetilde D(t)} > t/2 \right) \le P^0_{\M} \left(X_{t/2} \in B(0, g(t))\right) \le O(t^{-1-a}), 
\end{equation} 
and, 
\begin{align}\label{M-return}  
\sup_{y \in \R^d  \setminus \widetilde D(t)} P_{\M}^{y} \left(T_{D} < +\infty\right) 
&= \sup_{y \in \R^d \setminus \widetilde D(t)} P_{\textup{BM}}^{y} \left(T_{D} < +\infty\right) \notag \\
&\le O(g(t)^{2-d}) = O(t^{-1-a}). 
\end{align}

Henceforth we fix $h \in (0, 1)$.    
For $y \in \R^d \setminus D$, let 
\[ F_{\M} \left(y, s\right) := E_{\M}^y \left[ \left| W_{s, s+h} \setminus W_{0,s} \right|, \ T_D > s \right]. \]
We also define this quantity by replacing the case that we deal with $\M$ with the case that we deal with the standard Brownian motion. 
\[ F_{\textup{BM}} \left(y, s\right) := E_{\textup{BM}}^y \left[ \left| W_{s, s+h} \setminus W_{0,s} \right|, \ T_D > s \right]. \]
Furthermore we let 
\[ F_{\textup{BM}}(s) := E_{\textup{BM}}^0 \left[ \left| W_{s,s+h} \setminus W_{0,s} \right| \right]. \]

We will show that 
\begin{Lem}
\begin{equation}\label{sum-conv} 
\left| \sum_{n = 0}^{\infty} E^0_{\M} \left[\left|W_{nh, (n+1)h} \setminus W_{0,nh}\right| \right] - F_{\textup{BM}}(nh) \right| < +\infty.  
\end{equation} 
\end{Lem} 

\begin{proof} 
It holds that 
\[ E^0_{\M} \left[\left|W_{t, t+h} \setminus W_{0,t}\right| \right] \]
\[ = E^0_{\M} \left[\left|W_{t, t+h} \setminus W_{0,t}\right|, \tau_{\widetilde D(t)} > t \right] + E^0_{\M} \left[\left|W_{t, t+h} \setminus W_{0,t}\right|, \tau_{\widetilde D(t)} \le t \right] \]

By the Markov property, \eqref{0-1-finite} and \eqref{M-exit}, 
we have that 
\begin{equation}\label{step-1} 
E^0_{\M} \left[\left|W_{t, t+h} \setminus W_{0,t}\right|, \tau_{\widetilde D(t)} > t \right] \le \sup_{y \in \R^d} E^y_{\M} \left[\left|W_{0,h}\right|\right] = O(t^{-1-a}). 
\end{equation}

Let 
\[ T_{D, \widetilde D(t)} := \inf\left\{s > \tau_{\widetilde D(t)} : X_s \in D \right\}. \] 
Then it holds that  
\[ E^0_{\M} \left[ \left|W_{t, t+h} \setminus W_{0,t}\right|, \ \tau_{\widetilde D(t)} \le t \right] \]
\[= E^0_{\M} \left[ \left|W_{t, t+h} \setminus W_{0,t}\right|, \ \tau_{\widetilde D(t)} \le t < \tau_{\widetilde D(t)} + T_{D, \widetilde D(t)} \right] \]
\begin{equation}\label{step-2} 
+ E^0_{\M} \left[ \left|W_{t, t+h} \setminus W_{0,t}\right|, \ \tau_{\widetilde D(t)} + T_{D, \widetilde D(t)} \le t \right]. 
\end{equation}

By the Markov property and \eqref{M-return},   
\[ E^0_{\M} \left[\left|W_{t, t+h} \setminus W_{0,t}\right|, \ \tau_{\widetilde D(t)} + T_{D, \widetilde D(t)} \le t \right] \] 
\[ \le \sup_{y \in \R^d} E^y_{\M} \left[\left|W_{0,h}\right|\right] P^0_M \left(\tau_{\widetilde D(t)} + T_{D, \widetilde D(t)} \le t \right) \]
\begin{equation}\label{step-3} 
 \le \sup_{y \in \R^d} E^y_{\M} \left[\left|W_{0,h}\right|\right] \sup_{y \in \overline{\partial} \widetilde D(t)} P_{\M}^{y} (T_{D} \le t) = O(t^{-1-a}). 
\end{equation}

By the strong Markov property, 
we have that 
\[ E^0_{\M} \left[ \left|W_{t, t+h} \setminus W_{0,t}\right|, \ \tau_{\widetilde D(t)} \le t < \tau_{\widetilde D(t)} + T_{D, \widetilde D(t)} \right] \] 
\[ = E^0_{\M} \left[ \left|W_{t, t+h} \setminus W_{\tau_{\widetilde D(t)}, t}\right|, \ \tau_{\widetilde D(t)} \le t < \tau_{\widetilde D(t)} + T_{D, \widetilde D(t)} \right] \]
\[ = E^0_{\M} \left[ F_{\M} \left(X_{\tau_{\widetilde D(t)}}, t - \tau_{\widetilde D(t)}\right), \ \tau_{\widetilde D(t)} \le t\right]. \]

By Definition \ref{bddm-def}, 
\begin{equation}\label{step-4}  
E^0_{\M} \left[ F_{\M} \left(X_{\tau_{\widetilde D(t)}}, t - \tau_{\widetilde D(t)}\right), \ \tau_{\widetilde D(t)} \le t\right] = E^0_{\M} \left[ F_{\textup{BM}} \left(X_{\tau_{\widetilde D(t)}}, t - \tau_{\widetilde D(t)}\right), \ \tau_{\widetilde D(t)} \le t\right]. 
\end{equation}

By using the Markov property, the translation invariance of the law of the Brownian motion, $h < 1$, and \eqref{return-from-remote}, 
\[ \sup_{s \in (0,t], y \in \R^d \setminus \widetilde D(t)}  \left| F_{\textup{BM}}(y,s) - F_{\textup{BM}}(s)\right| 
= \sup_{s \in (0,t], y \in \R^d \setminus \widetilde D(t)} E_{\textup{BM}}^y \left[ \left| W_{s, s+h} \setminus W_{0,s} \right|, \ T_D \le s \right]\]
\[ \le \sup_{s \in (0,t], y \in \R^d \setminus \widetilde D(t)} E_{\textup{BM}}^y \left[ \left| W_{s, s+h} \right|, \ T_D \le s \right]\]
\[ \le E_{\textup{BM}}^0 \left[ \left| W_{0,1} \right| \right] \sup_{y \in \R^d \setminus \widetilde D(t)} P^y \left( T_D \le t \right) = O(t^{-1-a}).  \]
Therefore,
\[ E^0_{\M} \left[  \left| F_{\textup{BM}} \left(X_{\tau_{\widetilde D(t)}}, t - \tau_{\widetilde D(t)}\right) - F_{\textup{BM}} \left(t - \tau_{\widetilde D(t)}\right) \right|, \ \tau_{\widetilde D(t)} \le t\right] = O(t^{-1-a}). \]

By this, \eqref{step-1}, \eqref{step-2}, \eqref{step-3} and \eqref{step-4}, 
\begin{equation}\label{step-5}
\left| E^0_{\M} \left[\left| W_{t,t+h} \setminus W_{0,t} \right| \right] - E^0_{\M} \left[ F_{\textup{BM}}\left(t-\tau_{\widetilde D(t)}\right), \ \tau_{\widetilde D(t)} < t \right] \right| = O(t^{-1-a}). 
\end{equation}

By \cite{Sp64}, $F_{\textup{BM}}(s)$ is non-negative and non-increasing with respect to $s$. 
In particular $$\sup_{s \ge 0} F_{\textup{BM}}(s) \le F_{\textup{BM}}(0) \le E^0_{\M}[|W_{0,h}|]< +\infty.$$ 
Using this and \eqref{M-exit},  
\begin{align}\label{final-upper} 
F_{\textup{BM}}(t) - O(t^{-1-a}) &\le E^0_{\M} \left[ F_{\textup{BM}}(t), \ \tau_{\widetilde D(t)} \le \frac{t}{2} \right] \notag \\
&\le E^0_{\M} \left[ F_{\textup{BM}}(t), \ \tau_{\widetilde D(t)} < t\right] \notag \\
&\le E^0_{\M} \left[ F_{\textup{BM}}\left(t-\tau_{\widetilde D(t)}\right), \ \tau_{\widetilde D(t)} < t\right]   \notag\\
&\le E^0_{\M} \left[ F_{\textup{BM}}\left(t-\tau_{\widetilde D(t)}\right), \ \tau_{\widetilde D(t)} \le \frac{t}{2} \right] + E^0_{\M} \left[ F_{\textup{BM}}\left(t-\tau_{\widetilde D(t)}\right), \ \frac{t}{2} < \tau_{\widetilde D(t)} < t\right]   \notag\\
&\le E^0_{\M} \left[F_{\textup{BM}}(t/2), \tau_{\widetilde D(t)} < t\right] + O(t^{-1-a}) \notag\\ 
&\le F_{\textup{BM}}(t/2) + O(t^{-1-a}).  
\end{align} 

By \eqref{step-5} and \eqref{final-upper}, 
there exists a large positive constant $C$ such that for every $n \ge 1$, 
\begin{equation}\label{final-upper-2}
 F_{\textup{BM}}(nh) - C(nh)^{-1-a} \le E^0_{\M} \left[\left| W_{nh, (n+1)h} \setminus W_{0,nh} \right| \right] \le F_{\textup{BM}}(nh/2) + C(nh)^{-1-a}
\end{equation}

By \cite{Sp64} and the assumption that $d \ge 6$, 
it holds that 
\begin{equation}\label{spitzer}  
\left| \sum_{n = 0}^{\infty} h \ca\left(\overline{B}(0,\epsilon)\right) - F_{\textup{BM}}(nh) \right| < +\infty.  
\end{equation}
This implies that 
\[ \lim_{t \to \infty} F_{\textup{BM}}(t) = h \ca\left(\overline{B}(0,\epsilon)\right). \]
In particular $F_{\textup{BM}}(t) \ge h \ca\left(\overline{B}(0,\epsilon)\right)$ for every $t > 0$. 
By this and \eqref{spitzer}, we have that 
\begin{equation*}\label{final-upper-3} 
\int_{0}^{\infty} F_{\textup{BM}}(t/2) - h\ca\left(\overline{B}(0, \epsilon)\right) \ dt \le \int_{0}^{\infty} F_{\textup{BM}}(t/2)  - h\ca\left(\overline{B}(0, \epsilon)\right) dt < +\infty. 
\end{equation*}

Now \eqref{sum-conv} follows from this, \eqref{step-5}, and \eqref{final-upper-2}.  
\end{proof}

By \eqref{sum-conv} and \eqref{spitzer}, 
\begin{equation}\label{pre-final} 
\lim_{n \to \infty} E_{\M}^0 \left[ \left|W_{0,nh}\right| \right] - nh \ca\left(\overline{B}(0, \epsilon)\right)\textup{ exists and is finite.} 
\end{equation}

Since $|W_{0,t}|$ is non-decreasing with respect to $t$, 
we have that for every $n \ge 0$ and $t \in [nh, (n+1)h]$, 
\[ \left| \left( E_{\M}^0 \left[ \left|W_{0,t}\right| \right] - t \ca\left(\overline{B}(0, \epsilon)\right) \right) - \left( E_{\M}^0 \left[ \left|W_{0,nh}\right| \right] - nh \ca\left(\overline{B}(0, \epsilon)\right) \right) \right| \]
\[ \le E^0_{\M} \left[\left| W_{nh, (n+1)h}  \setminus W_{0, nh} \right| \right] + h \ca (\overline{B}(0, \epsilon)). \]

By this and \eqref{pre-final}, we have that 
\[ \limsup_{t \to \infty} \left\{E_{\M}^0 \left[ \left|W_{0,t}\right| \right] - t \ca\left(\overline{B}(0, \epsilon)\right)\right\} - \liminf_{t \to \infty} \left\{E_{\M}^0 \left[ \left|W_{0,t}\right| \right] - t\ca\left(\overline{B}(0, \epsilon)\right) \right\} \]
\[ \le h \ca\left(\overline{B}(0, \epsilon)\right) + \limsup_{n \to \infty} E^0_{\M} \left[ \left|W_{nh, (n+1)h} \setminus W_{0, nh} \right| \right]. \] 

Therefore it suffices to show that 
\[ \lim_{h \to 0+} \limsup_{n \to \infty} E^0_{\M} \left[ \left| W_{nh, (n+1)h} \setminus W_{0, nh}\right| \right] = 0. \] 

In order to show this, it suffices to show that 
\begin{equation}\label{final} 
\limsup_{n \to \infty} E^0_{\M} \left[ \left|W_{nh, (n+1)h} \setminus W_{0, nh} \right| \right] \le E^0_{\textup{BM}} \left[ \left|W_{0, h} \setminus W_{0,0} \right| \right]. 
\end{equation}  
Here we remark that $W_{0,0} = B(0, \epsilon)$.  

Let $\delta > 0$. 
Then, there exists $R(\delta, h) > 0$ such that 
\begin{equation}\label{final-R} 
P^0_{\textup{BM}}\left(\tau_{B(0, R(\delta, h))} \le h\right) \le \delta. 
\end{equation}

By FHK$(2)$, if $n$ is sufficiently large, then, 
\[ P^0_{\M} \left(X_{nh} \in B(0, 2R(\delta, h) + R_0) \right) \le \delta. \]
By this and the Markov property, 
\begin{equation}\label{final-1}  
E^0_{\M} \left[ \left| W_{nh, (n+1)h}\right|, \  X_{nh} \in B(0, 2R(\delta, h) + R_0) \right] \le \delta \sup_{y \in \mathbb{R}^d} E^y_{\M} \left[ \left| W_{0, h}\right| \right]. 
\end{equation}

Let \[ \tau_n := \inf\{s > nh : X_s \notin B(X_{nh}, R(\delta, h))\}.\] 
Then, 
\[ \left| W_{nh, (n+1)h} \setminus W_{0, nh} \right| \le \left| W_{nh,  \tau_n \wedge (n+1)h} \setminus W_{0, nh} \right| + \left|W_{\tau_n \wedge (n+1)h, (n+1)h}  \setminus W_{0, \tau_n \wedge (n+1)h} \right|\]
\[ \le \left| W_{nh,  \tau_n \wedge (n+1)h} \setminus W_{nh, nh} \right| + \left| W_{\tau_n \wedge (n+1)h, (n+1)h} \setminus W_{\tau_n \wedge (n+1)h, \tau_n \wedge (n+1)h} \right|. \]

By the Markov property, the definition of $R_0$ and Definition \ref{bddm-def}, 
we have that 
\[ E^0_{\M} \left[ \left| W_{nh,  \tau_n \wedge (n+1)h} \setminus W_{nh, nh} \right|, \ X_{nh} \notin B(0, 2R(\delta, h) + R_0) \right]  \]
\[ = E^0_{\M} \left[ E^{X_{nh}}_{\M} \left[  \left| W_{0,  \tau \wedge h} \setminus W_{0, 0} \right|  \right], \ X_{nh} \notin B(0, 2R(\delta, h) + R_0) \right] \]
\[ = E^0_{\M} \left[ E^{0}_{\textup{BM}} \left[  \left| W_{0,  \tau \wedge h} \setminus W_{0, 0} \right|  \right], \ X_{nh} \notin B(0, 2R(\delta, h) + R_0) \right] \]
\begin{equation}\label{final-2}  
\le E^{0}_{\textup{BM}} \left[  \left| W_{0, h} \setminus W_{0, 0} \right|  \right], 
\end{equation}
where in the first equality we let \[ \tau := \inf\{s > 0 : X_s \notin B(X_{0}, R(\delta, h))\}. \] 

By the strong Markov property, 
\[ E^0_{\M} \left[ \left| W_{\tau_n \wedge (n+1)h, (n+1)h} \setminus W_{\tau_n \wedge (n+1)h, \tau_n \wedge (n+1)h} \right|, \ X_{nh} \notin B(0, 2R(\delta, h) + R_0) \right]  \]
\[ \le E^0_{\M} \left[ \left| W_{\tau_n, (n+1)h} \right|, \ \tau_n < (n+1)h, \ X_{nh} \notin B(0, 2R(\delta, h) + R_0) \right]\]
\[ =  E^0_{\M} \left[  E^{X_{nh}}_{\M} \left[ \left| W_{\tau, h} \right|, \ \tau < h\right], \ X_{nh} \notin B(0, 2R(\delta, h) + R_0) \right]\]
\[ \le  E^0_{\M} \left[  E^{X_{nh}}_{\M} \left[  E^{X_{nh + \tau}}_{\M} \left[\left| W_{0, h} \right|\right], \ \tau < h\right], \ X_{nh} \notin B(0, 2R(\delta, h) + R_0) \right]\]
\[ \le \sup_{y \in \mathbb{R}^d} E^y_{\M} \left[ \left| W_{0, h}\right| \right] E^0_{\M} \left[  P^{X_{nh}}_{\M} \left(  \tau < h\right), \ X_{nh} \notin B(0, 2R(\delta, h) + R_0) \right]. \]

If $X_{nh} \notin B(0, 2R(\delta, h) + R_0)$, then by the definition of $R_0$, Definition \ref{bddm-def}, and \eqref{final-R}, 
\[ P^{X_{nh}}_{\M} \left(  \tau < h\right) = P^{0}_{\textup{BM}} \left(  \tau < h\right) \le \delta. \]

Hence, 
\[ E^0_{\M} \left[ \left| W_{\tau_n \wedge (n+1)h, (n+1)h} \setminus W_{\tau_n \wedge (n+1)h, \tau_n \wedge (n+1)h} \right|, \ X_{nh} \notin B(0, 2R(\delta, h) + R_0) \right] \]
\begin{equation}\label{final-3}   
\le\delta \sup_{y \in \mathbb{R}^d} E^y_{\M} \left[ \left| W_{0, h}\right| \right]. 
\end{equation}

By \eqref{final-1}, \eqref{final-2} and \eqref{final-3}, we have that 
\[ E^0_{\M} \left[ \left| W_{nh, (n+1)h} \setminus W_{0, nh} \right| \right] \le 2\delta \sup_{y \in \mathbb{R}^d} E^y_{\M} \left[ \left| W_{0, h}\right| \right] +  E^0_{\textup{BM}} \left[ \left|W_{0, h} \setminus W_{0,0} \right| \right]. \]
Since $\delta > 0$ is taken arbitrarily, 
we have \eqref{final}. 
Thus we have the assertion. 
\end{proof}


\appendix

\section{Several claims used in the proof of Theorem \ref{pre-SLLN}}

In this section we state several claims used in the proof of Theorem \ref{pre-SLLN}. 
Most proofs depend on \cite{KKW17}. 
We remark that \cite{KKW17} deals with jump processes, however it holds  also for diffusions satisfying $\textup{Vol}(V; \alpha_1, \alpha_2)$ and $\textup{HK}(\phi; \beta_1, \beta_2)$.    

\subsection{0-1 law}

Throughout this subsection, we assume that $\textup{Vol}(V; \alpha_1, \alpha_2)$ and HK$(\phi; \beta_1, \beta_2)$ hold for certain $V$ and $\phi$ satisfying \eqref{V-alpha} and \eqref{phi-beta} respectively. 
 
\begin{Thm}[The Barlow-Bass zero-one law {\cite[Theorem 8.4]{BB99}}]\label{tail-01} 
Let $A$ be a tail event, that is, $A \in \cap_{t > 0} \sigma(\{X_u, u \ge t\})$. 
Then, 
either $P^x (A) = 0$ holds for every $x \in M$, or, $P^x (A) = 1$ holds for every $x \in M$.  
\end{Thm}

By \cite[Corollary 5.12 and Theorem 2.12]{GT12},  we have that  
\begin{Prop}[H\"older continuity for the heat kernel]\label{Holder}
There exist two positive constants $C$ and $\theta$ such that 
\[ \left|p(t,x,x) - p(t,x,y)\right| \le \frac{C}{V(\phi^{-1}(t))} \left(\frac{d(x,y)}{\phi^{-1}(t)}\right)^{\theta}, \ \ d(x,y) \le \phi^{-1}(t). \]
\end{Prop}

Since we assume HK$(\phi; \beta_1, \beta_2)$, our framework is a little more general than \cite{BB99}. 
However, we can show Theorem \ref{tail-01} as in \cite[Theorem 8.4]{BB99} once we have Proposition \ref{Holder} above, 
so we do not give the proof. 
We remark that we need the uniform volume growth condition \eqref{UVG} for the proof, and hence Theorem \ref{tail-01} is not applicable to the framework of \cite[Subsection 4.3]{GSC05}.  
Informally this occurs due to the lack of homogeneity of the framework of \cite[Subsection 4.3]{GSC05}. 
See also Subsection 2.1.

\subsection{Chung-type liminf laws of the iterated logarithm}
 
Throughout this subsection,  we assume that $\textup{Vol}(V; \alpha_1, \alpha_2)$ and HK$(\phi; \beta_1, \beta_2)$ hold for certain $V$ and $\phi$ satisfying \eqref{V-alpha} and \eqref{phi-beta} respectively. 

\begin{Thm}[cf. {\cite[Theorem 3.18]{KKW17}}]\label{318-KKW} 
There exists a constant $C$ such that for every $x \in M$,  
\[ \liminf_{t \to \infty} \frac{\sup_{s \le t} d(X_0, X_s)}{\phi^{-1}(t /\log\log t)} \le C,  \textup{ $P^x$-a.s.}   \]
\end{Thm}

We remark that we do not need to assume FHK$(\phi; \beta_1, \beta_2)$ above. 
The proof  of Theorem \ref{318-KKW}  goes in the same manner as in the proof of \cite[Theorem 3.18]{KKW17} if we have the following two assertions.

\begin{Lem}[Estimate for exit time; {cf. \cite[(3.4)]{KKW17}}] \label{(3.4)-KKW}
There exists a positive constant  $c$ such that for any $x \in M$ and any $t, r > 0$, 
\[ P^x \left(\sup_{s \le t} d(x, X_s) \ge r\right) \le c\frac{t}{\phi(r)}. \]
\end{Lem}

The assumption for the upper bound of the heat kernel \eqref{HK-upper} is sufficient to show this. 
The proof goes in the same manner as in the proof of \cite[(3.4)]{KKW17}. 

\begin{Prop}[cf. {\cite[Proposition 2.12]{KKW17}}]\label{(2.12)-KKW}
Assume that $\textup{Vol}(V; \alpha_1, \alpha_2)$ and HK$(\phi; \beta_1, \beta_2)$. 
Then, there exist a positive constant $c$ and $a_1^*, a_2^* \in (0,1)$ such that 
\[ {(a_1^*)}^n \le P^x \left(\sup_{s \in [0, c \phi(r) n]} d(x, X_s) \le r \right) \le (a_2^*)^n. \]
holds for every $n \ge 1$, $r > 0$ and $x \in M$. 
\end{Prop}

The proof goes in a manner similar to the proof of \cite[Proposition 2.12]{KKW17}. 

\begin{Rem}\label{211-KKW}
The proof of \cite[Proposition 2.12]{KKW17} depend on  \cite[Proposition 2.11]{KKW17}. 
The proof of \cite[Proposition 2.11]{KKW17} uses a chaining argument of \cite[Lemma 2.3]{BBK09}. 
The chaining argument is done under the assumption that the metric is {\it geodesic}. 
However, \cite[Proposition 2.12]{KKW17} does not assume that the metric is geodesic. 
This does not imply any major problems, because it is easy to show  \cite[Proposition 2.12]{KKW17} without using the {\it full} statement of \cite[Proposition 2.11]{KKW17} but using a {\it partial} version of \cite[Proposition 2.11]{KKW17}, which is obtained by replacing $B(x, r/2)$ in  \cite[Proposition 2.11]{KKW17} with $B(x, \delta_0 r/2)$ where $\delta_0$ denotes the same constant as in  \cite[Proposition 2.11]{KKW17}. 
\end{Rem} 

\subsection{Laws of iterated logarithms for local time}

Throughout this subsection, we assume that Vol$(V; \alpha_1 \alpha_2)$ and FHK$(\phi; \beta_1, \beta_2)$ hold for certain $V$ and $\phi$ satisfying \eqref{V-alpha} and \eqref{phi-beta} respectively, 
and furthermore $\alpha_2 < \beta_1$. 
Under these assumptions we have \eqref{converge-int},  and in the same manner as in the proof of  \cite[Proposition 4.3]{KKW17}, 
we can show that a version of the local time of $X$ exists, specifically, there exists a random field $\ell (t,x)(\omega)$ satisfying the following conditions: 
(1) $\ell (t,x)(\omega)$ is jointly measurable with respect to $(t,x,\omega)$, and 
\[ \int_0^t h(X_s) ds = \int_{M} h(x) \ell (t,x) \mu(dx) \]
holds for every $t > 0$ and every Borel measurable function $h$ on $M$. \\
(2) We have that for every $\lambda > 0$ and every $x, y \in M$, 
\[ E^x \left[ \int_0^{+\infty} \exp(-\lambda t) d\ell(t,y)\right] = \int_0^{+\infty} \exp(-\lambda t) p(t,x,y) dt, \] 
where for every fixed $y in M$, $d\ell(t,y)$ denotes the measure with respect to $t$.\\
These conditions correspond to \cite[property (1) and (2) of Proposition 4.3]{KKW17} respectively.   

\begin{Thm}[cf. {\cite[Upper bounds for Theorems 4.11 and 4.15]{KKW17}}]\label{411415} 
There exist two constants $c_{\sup}$ and $c_{\inf}$ such that for every $x \in M$, \\
(i)
\[ \limsup_{t \to \infty} \frac{\sup_{z \in M} \ell(t,z)}{t/V(\phi^{-1}(t/\log\log t))} \le c_{\sup}, \textup{ $P^x$-a.s.}  \]
(ii) 
\[ \liminf_{t \to \infty} \frac{\sup_{z \in M} \ell(t,z)}{(t/\log\log t)/V(\phi^{-1}(t/\log\log t))} \le c_{\inf}, \textup{ $P^x$-a.s.}  \]
\end{Thm}

The most important step of the proof  is establishing the following assertion. 

\begin{Prop}[Modified statement of {\cite[Proposition 4.8]{KKW17}}]\label{Pp4.8-KKW}
There exist two positive constants $c_1, c_2$ such that for every $A, L, u > 0$, 
\[ P^z \left(\sup_{d(x,y) \le L} \sup_{t \in [0, u]} |\ell(t,x) - \ell(t,y)| > A \right) \]
\[\le c_1 \frac{V(\max\{L, \phi^{-1}(u)\})^2}{V(L)^2} \exp\left(-c_2 A \left(\frac{V(L) V(\max\{L, \phi^{-1}(u)\})}{\phi(L) \phi(\max\{L, \phi^{-1}(u)\})}\right)^{1/2}\right). \]
\end{Prop}

Given this assertion, Theorem \ref{411415} can be proved similarly to \cite[Proof of Theorems 4.11 and 4.15]{KKW17}. 
The statement is changed from \cite[Proposition 4.8]{KKW17}, accompanying a modification \eqref{U-delta-prime} in the proof below.

\begin{proof}[Outline of Proof of Proposition \ref{Pp4.8-KKW}]
The proof goes in the same manner as in the proof of \cite[Proposition 4.8]{KKW17} with minor modifications. 

For $\delta > 0$, we define the following: 
$d_{(\delta)}(x,y) := d(x,y)/\delta$. 
$\mu_{(\delta)}(A) := \mu(A)/V(\delta)$. 
Let 
$X^{(\delta)}_t := X_{\phi(\delta) t}$. 
Then, 
\[ p_{(\delta)}(t,x,y) = V(\delta) p(\phi(\delta)t, x, y).\]  
Let 
\[ V_{(\delta)}(r) := V(\delta r) / V(\delta), \ \phi_{(\delta)}(r) := \phi(\delta r) / \phi(\delta).  \]
Then, 
\[ \frac{V(\delta)}{V(\phi^{-1}(\phi(\delta)t))} = \dfrac{V(\delta)}{V\left(\delta \dfrac{\phi^{-1}(\phi(\delta)t)}{\delta}\right)} =  \dfrac{1}{V_{(\delta)}\left( \dfrac{\phi^{-1}(\phi(\delta)t)}{\delta}\right)} = \frac{1}{V_{(\delta)}(\phi_{(\delta)}^{-1}(t))}. \]
By \eqref{V-alpha} and \eqref{phi-beta}, we also have that
\[ \left(\frac{R}{r}\right)^{\alpha_1} \le \frac{V_{(\delta)} (R)}{V_{(\delta)} (r)} \le \left(\frac{R}{r}\right)^{\alpha_2}, \ 0 < r < R, \]
and, 
\[ \left(\frac{R}{r}\right)^{\beta_1} \le \frac{\phi_{(\delta)} (R)}{\phi_{(\delta)} (r)} \le \left(\frac{R}{r}\right)^{\beta_2}, \ 0 < r < R. \]

Let 
\[ \Psi_{(\delta)} (r, t) = \sup_{s > 0} \frac{r}{s} - \frac{t}{\phi_{\delta}(s)}.  \]
Then  
\[ \Psi (d(x,y), \phi(\delta)t) = \Psi_{(\delta)}(d_{(\delta)(x,y)}, t). \]

By these results and \eqref{HK-lower} and \eqref{HK-upper}, we have that for every $\delta > 0$, 
\[ \frac{c_5}{ \mu_{(\delta)}(B_{(\delta)}(x,\phi^{-1}(t)))} \le p_{(\delta)}(t,x,y), \ d_{(\delta)}(x,y) \le c_6\phi_{(\delta)}^{-1}(t), \]
and,   
\[ p_{(\delta)}(t,x,y) \le  \frac{c_7 \exp\left(-c_8 \Psi_{(\delta)}(d_{(\delta)}(x,y), t)  \right)}{ \mu_{(\delta)}(B_{(\delta)}(x,\phi_{(\delta)}^{-1}(t)))}, \ x,y \in M, t > 0, \]
where $c_5, c_6, c_7$ and $c_8$ are  the same constants as in \eqref{HK-lower} and \eqref{HK-upper}.

Let $\ell^{(\delta)}(t,x)$ be a local time for the scaled process $X^{(\delta)}$ with respect to $\mu_{(\delta)}$.  
We assume that it satisfies \cite[property (1) and (2) of Proposition 4.3]{KKW17} for the scaled setting and furthermore it is jointly continuous almost surely. 
Let $P_{(\delta)}^x$ be the probability law of $X^{(\delta)}$ starting at $x \in M$. 

Henceforth, we let $\delta^{\prime} := 1/\delta$. 
Then for $\delta = \min\{1/\phi^{-1}(u), 1/L\}$, 
\[ P^z \left(\sup_{d(x,y) \le L} \sup_{t \in [0, u]} |\ell(t,x) - \ell(t,y)| > A \right) \]
\[ \le P_{(\delta^{\prime})}^z \left(\sup_{d_{(\delta^{\prime})} (x,y) \le \delta L} \sup_{t \in [0, 1]} \left|\ell^{(\delta^{\prime})}(t,x) - \ell^{(\delta^{\prime})}(t,y)\right| > A\frac{V(\delta^{\prime})}{\phi(\delta^{\prime})}  \right). \]
See \cite[(4.20)]{KKW17} for details. 

Let 
\begin{equation}\label{U-delta-prime}
U_{(\delta^{\prime})}(r) := \sqrt{\frac{\phi_{(\delta^{\prime})}(r)}{V_{(\delta^{\prime})}(r)}}.
\end{equation}  
(This part is different from the proof of \cite[Proposition 4.8]{KKW17}.)
Let
\[ F_{\delta^{\prime}} :=  \int\int_{d_{(\delta^{\prime})}(x,y) \le 1} \left(\exp\left(c_{*} \frac{\sup_{t \in [0, 1]} \left|\ell^{(\delta^{\prime})}(t,x) - \ell^{(\delta^{\prime})}(t,y)\right| }{U_{(\delta^{\prime})}(d_{(\delta^{\prime})}(x,y))} \right) - 1\right) \mu_{(\delta^{\prime})}(dx)\mu_{(\delta^{\prime})}(dy) \]
for a sufficiently small constant $c_{*} > 0$. 

We assume that 
\begin{equation}\label{unif-delta-prime}
\sup_{\delta > 0} E_{(\delta^{\prime})}\left[F_{\delta^{\prime}}\right]< +\infty.
\end{equation} 
By Garsia's lemma stated in \cite[Lemma A.1]{KKW17}, 
there exists two positive constants $c_3$ and $c_4$ independent from $\delta$ such that  
\[ \left|\ell^{(\delta^{\prime})}(t,x) - \ell^{(\delta^{\prime})}(t,y)\right| \le c_3 U_{(\delta^{\prime})}(\delta L) \log\left(1 + c_4 \frac{F_{\delta^{\prime}}}{V_{(\delta^{\prime})}(\delta L)^2}\right) \]
holds for $\mu_{(\delta^{\prime})}$-a.e. $x, y \in M$ satisfying that $d_{(\delta^{\prime})}(x,y) \le L$ and $t \in [0,1]$. 

By using this and the joint continuity of local time, we can derive that 
\[ P_{(\delta^{\prime})}^z \left(\sup_{d_{(\delta^{\prime})} (x,y) \le \delta L} \sup_{t \in [0, 1]} \left|\ell^{(\delta^{\prime})}(t,x) - \ell^{(\delta^{\prime})}(t,y)\right| > A\frac{V(\delta^{\prime})}{\phi(\delta^{\prime})}  \right) \]
\[ \le \frac{c_5}{V_{(\delta^{\prime})}(\delta L)^2} \exp\left(-\frac{A V(\delta^{\prime})}{c_3  U_{(\delta^{\prime})}(\delta L) \phi(\delta^{\prime})}\right) \left(1 + c_4 \frac{E_{(\delta^{\prime})} [F_{\delta^{\prime}}]}{V_{(\delta^{\prime})}(\delta L)^2}\right)\]
Hence the assertion follows once we show \eqref{unif-delta-prime}. 

We can show \eqref{unif-delta-prime} in the same manner as in the proof of \cite[Proposition 4.8]{KKW17}.
In the proof, \cite[Proposition 2.11]{KKW17} is used. However, the issue raised in Remark \ref{211-KKW} does not affect any estimates in our proof. 
\end{proof}

\vspace{1\baselineskip}

\noindent\textit{Acknowledgements.} 
The author wishes to give his thanks to the referee for careful reading of the manuscript and giving many comments. 
He also wishes to give his thanks to Kazumasa Kuwada for comments on the modification of a metric measure space and to Jian Wang for the above correction of Proposition \ref{Pp4.8-KKW} in Appendix. 
This work was supported by JSPS KAKENHI Grant-in-Aid for JSPS Research Fellows (16J04213) and by the Research Institute for Mathematical Sciences, a Joint Usage/Research Center located in Kyoto University.   
 
\vspace{1\baselineskip}


\begin{thebibliography}{99} 
\bibitem[B04]{B04} M. T. Barlow, Anomalous diffusion and stability of Harnack inequalities. Surveys in differential geometry. Vol. IX, 1-25, Surv. Differ. Geom., IX, Int. Press, Somerville, MA, 2004. 
\bibitem[BB92]{BB92} M. T. Barlow and R. F. Bass, Transition densities for Brownian motion on the Sierpinski carpet. {\it Probab. Theory Related Fields} 91 (1992), 307-330.
\bibitem[BB99]{BB99} M. T. Barlow and R. F. Bass, Brownian motion and harmonic analysis on Sierpinski carpets, {\it Canadian Journal of Math.} 51 (1999) 673-744.
\bibitem[BB00]{BB00} M. T. Barlow and R. F. Bass, Divergence Form Operators on Fractal-like Domains, {\it J. Funct. Anal.} 175 (2000) 214-247. 
\bibitem[BBK06]{BBK06} M. T. Barlow, R. F. Bass and T. Kumagai, Stability of parabolic Harnack inequalities on metric measure spaces, {\it J. Math. Soc. Japan} 58 (2006), 485-519.
\bibitem[BBK09]{BBK09} M. T. Barlow, R. F. Bass and T. Kumagai, Parabolic Harnack inequality and heat kernel estimates for random walks with long range jumps, {\it Math. Z.} 261 (2009) 297-320.
\bibitem[BGK12]{BGK12} M. T. Barlow, A. Grigor'yan and T. Kumagai, On the equivalence of parabolic Harnack inequalities and heat kernel estimates, {\it J. Math. Soc. Japan} 64 (2012) 1091-1146. 
\bibitem[BP88]{BP88} M. T. Barlow and E. A. Perkins, Brownian motion on the Sierpinski gasket. {\it Probab. Theory Related Fields} 79 (1988), 543-623.
\bibitem[BK00]{BK00} R. F. Bass, and T. Kumagai, Laws of the iterated logarithm for some symmetric diffusion processes. {\it Osaka J. Math.} 37 (2000) 625-650.
\bibitem[CKS87]{CKS} E. A. Carlen, S. Kusuoka, S and D. W. Stroock, Upper bounds for symmetric Markov transition functions. {\it Ann. Inst. H. Poincar\'e Probab. Statist.} 23 (1987), no. 2, suppl., 245-287.
\bibitem[CF86-1]{CF86-1}  I. Chavel and E. A. Feldman, The Wiener sausage and a theorem of Spitzer in Riemannian manifolds. {\it Probability theory and harmonic analysis (Cleveland, Ohio, 1983)}, 45-60, Monogr. Textbooks Pure Appl. Math., 98, Dekker, New York, 1986.
\bibitem[CF86-2]{CF86-2} I. Chavel and E. A. Feldman, The Lenz shift and Wiener sausage in Riemannian manifolds. {\it Compositio Math.} 60 (1986), no. 1, 65-84.
\bibitem[CF86-3]{CF86-3} I. Chavel and E. A. Feldman, The Lenz shift and Wiener sausage in insulated domains. {\it From local times to global geometry, control and physics (Coventry, 1984/85)}, 47-67, Pitman Res. Notes Math. Ser., 150, Longman Sci. Tech., Harlow, 1986.
\bibitem[CFR91]{CFR91}  I. Chavel, E. A. Feldman and J. Rosen, Fluctuations of the Wiener sausage for surfaces. {\it Ann. Probab.} 19 (1991) 83-141.
\bibitem[DV75]{DV75} M. D. Donsker and S. R. S. Varadhan, Asymptotics for the Wiener sausage, {\it Comm. Pure Appl. Math.} 28 (1975) 525-565.
\bibitem[FOT11]{FOT11} M. Fukushima, Y. Oshima and M. Takeda, Dirichlet Forms and Symmetric Markov Processes, Second revised and extended edition. de Gruyter Studies in Mathematics, 19. Walter de Gruyter \& Co., Berlin, 2011. 
\bibitem[GP15]{GP15} L. R. Gibson and M. Pivarski, The Rate of Decay of the Wiener Sausage in Local Dirichlet Space, {\it J. Theor. Probab.} 28 (2015) 1253-1270. 
\bibitem[G97]{G97} A. Grigor'yan, Gaussian upper bounds for the heat kernel on arbitrary manifolds, {\it J. Differential Geom.} 45 (1997) 33-52. 
\bibitem[G99]{G99} A. Grigor'yan, Analytic and geometric background of recurrence and non-explosion of the Brownian motion on Riemannian manifolds, {\it Bull. Amer. Math. Soc.} 36 (1999) 135-249.
\bibitem[GSC02]{GSC02} A. Grigor'yan and L. Saloff-Coste, Hitting probabilities for Brownian motion on Riemannian manifolds. {\it J. Math. Pures Appl.}  (9), 81 (2002) 115-142.
\bibitem[GSC05]{GSC05} A. Grigor'yan and L. Saloff-Coste, Stability results for Harnack inequalities. Ann. Inst. Fourier (Grenoble) 55 (2005) 825-890. 
\bibitem[GT12]{GT12} A. Grigor'yan, and A. Telcs, Two-sided estimates of heat kernels on metric measure spaces, {\it Ann. Probab.} 40 (2012) 1212-1284.
\bibitem[HSC01]{HSC01} W. Hebisch and L. Saloff-Coste, On the relation between elliptic and parabolic Harnack inequalities, {\it Ann. Inst. Fourier (Grenoble)}, 51 (2001) 1437-1481.
\bibitem[KL74]{KL74} M. Kac and J. M. Luttinger, Bose Einstein condensation in the presence of impurities II, {\it J. Math. Phys.}, 15 (1974), 183-186.
\bibitem[Ka85]{Ka85} M. Kanai. Rough isometries, and combinatorial approximations of geometries of non-compact Riemannian manifolds, J. Math. Soc. Japan 37 (1985) 391-413.
\bibitem[KKW17]{KKW17} P. Kim, T. Kumagai, and J. Wang, Laws of the iterated logarithm for symmetric jump processes. {\it Bernoulli} 23 (2017) 2330-2379.  
\bibitem[Le88]{Le88} J.-F. Le Gall, Fluctuation results for the Wiener sausage, {\it Ann. Probab.} 16 (1988)  991-1018.
\bibitem[LY86]{LY86} P. Li and  S.-T. Yau, On the parabolic kernel of the Schr\"odinger operator, {\it Acta Math.} 156 (1986) 153-201.
\bibitem[N58]{N58} J. Nash, Continuity of solutions of parabolic and elliptic equations. {\it Amer. J. Math.} 80 (1958) 931-954. 
\bibitem[O14]{O14}  K. Okamura, On the range of random walk on graphs satisfying a uniform condition,  {\it ALEA Lat. Am. J. Probab. Math. Stat.} 11 (2014) 341-357.
\bibitem[O]{Opre} K. Okamura, Some results for range of random walk on graph with spectral dimension two, in preparation. 
\bibitem[ST92]{ST92} I. Shigekawa, and S. Taniguchi, Dirichlet forms on separable metric spaces. {\it Probability theory and mathematical statistics (Kiev, 1991)}, 324-353, World Sci. Publ., River Edge, NJ, 1992.
\bibitem[Sp64]{Sp64} F. Spitzer, Electrostatic capacity, heat flow and Brownian motion, {\it Z. Wahr. Verw. Gebiete}, 3 (1964) 110-121.
\bibitem[Sz89]{Sz89} A.-S. Sznitman, Lifschitz tail and Wiener sausage on hyperbolic space, {\it Comm. Pure Appl. Math.}, 42 (1989) 1033-1065.
\bibitem[Sz90]{Sz90} A.-S. Sznitman, Lifschitz tail and Wiener sausage. II. {\it J. Funct. Anal.} 94 (1990) 247-272. 
\end{thebibliography}
\end{document}